\documentclass[final]{siamltex704}

\usepackage{amsmath}
\usepackage{amssymb}
\usepackage{enumerate}
\usepackage{url}
\usepackage{graphicx}

\newcommand{\A}{{\mathcal A}}
\newcommand{\B}{{\mathcal B}}
\newcommand{\D}{{\mathcal D}}
\newcommand{\E}{{\mathcal E}}
\newcommand{\F}{{\mathcal F}}
\newcommand{\G}{{\mathcal G}}
\newcommand{\HH}{{\mathcal H}}

\renewcommand{\ldots}{\dotsc}

\title{On Extremal \lowercase{\textit{k}}-Graphs Without Repeated Copies of 2-Intersecting
Edges\thanks{Received by the
editors November 26, 2006; accepted for publication (in revised form) July 20,
2007; published electronically October 31, 2007.   \URL sidma/21-3/67591.html}}

\author{Yeow Meng Chee\thanks{Interactive Digital Media R\&D Program Office,
	Media Development Authority,
	140 Hill Street,
    179369    Singapore,
        the Division of Mathematical Sciences,
	School of Physical and Mathematical Sciences,
	Nanyang Technological University,
      637616 Singapore, and the Department of Computer Science, School of
      Computing, National University of Singapore, 117590 Singapore
       (ymchee@alumni.\break uwaterloo.ca).
	The research of this author was supported by the Singapore Ministry of Education
        Research Grant T206B2204.}
        \and Alan C. H. Ling\thanks{Department of Computer Science,
        University of VT,
        Burlington, VT 05405 (aling@emba.\break uvm.edu).}}

\begin{document}
\slugger{sidma}{2007}{21}{3}{805--821}
\maketitle

\setcounter{page}{805}

\begin{abstract}
The problem of determining extremal hypergraphs containing at most $r$ isomorphic copies of
some element of a given hypergraph family was first studied by Boros et al.\ in 2001. There
are not many hypergraph families for which exact results
are known concerning the size of the corresponding extremal hypergraphs,
except for those equivalent to the classical Tur\'an numbers. In this paper, we determine the
size of extremal $k$-uniform hypergraphs containing at most one pair of 2-intersecting edges
for $k\in\{3,4\}$.
We give a complete solution when $k=3$ and an almost complete
solution (with eleven exceptions) when $k=4$.
\end{abstract}



\begin{DOI}
10.1137/060675915
\end{DOI}

\pagestyle{myheadings}
\thispagestyle{plain}
\markboth{YEOW MENG CHEE AND ALAN C. H. LING}{\lowercase{\textit{k}}-GRAPHS WITHOUT REPEATED 2-INTERSECTING EDGES}

\section{Introduction}

A {\em set system} is a pair $G=(X,\A)$, where $X$ is a finite set and $\A\subseteq 2^X$.
The members of $X$ are called {\em vertices} or {\em points}, and the members of $\A$
are called {\em edges} or {\em blocks}. The {\em order} of $G$ is the number of vertices 
$|X|$, and the {\em size} of $G$ is the number of edges $|\A|$.
The set $K$ is called a {\em set of block sizes} for $G$ if $|A|\in K$ for all
$A\in\A$.
$G$ is called a $k$-{\em uniform hypergraph} (or $k$-{\em graph})
if $\{k\}$ is a set of block sizes for $G$. A
$2$-graph is also known simply as a {\em graph}.

A pair of edges is said to be $t$-{\em intersecting} if they intersect in at least $t$ points.
The $k$-graph of size two whose two edges intersect in exactly $t$ points is denoted
$\Lambda(k,t)$.

Let $\F$ be a family of $k$-graphs.
Boros et al.\ \cite{Borosetal:2001} introduced the function $T(n,\F,r)$, which denotes
the maximum number of edges in a $k$-graph of order $n$ containing no $r$
isomorphic copies of a member of $\F$. So $T(n,\F,1)$ is just the classical
Tur\'an number ex$(n,\F)$ \cite{Bollobas:1978}. A family of $k$-graphs $\F$ is said to
{\em grow polynomially} if there exist $c>0$ and a nonnegative integer $s$ such that,
for every $m$, there are at most $cm^s$ members in $\F$ having exactly $m$ edges.
The following theorem is established in \cite{Borosetal:2001}.

\begin{theorem}[Boros et al.\ \cite{Borosetal:2001}]\label{boros}
Let $\F$ be a family of $k$-graphs which grows polynomially with parameters $c$ and $s$.
Then, for $n$ sufficiently large,
\begin{align*}
T(n,\F,r) < & ~{\rm ex}(n,\F) + (c\cdot (r-1)\cdot s! + 1){\rm ex}(n,F)^{(s+1)/(s+2)} \\
& ~+ 2(c\cdot (r-1)\cdot s! + 1)^2 {\rm ex}(n,\F)^{s/(s+2)}.
\end{align*}
\end{theorem}\unskip

For $k\geq 3$, let $\F(k)$ be the family of $k$-graphs of two 2-intersecting edges; that is,
$\F(k)=\{\Lambda(k,t) : 2\leq t\leq k-1\}$. $T(n,\F(k),1)$, which is the Tur\'an number ex$(n,\F(k))$,
is equal to the following well studied parameters in design theory and coding theory:
\begin{itemize}
\item $D(n,k,2)$, the maximum number of blocks in a 2-$(n,k,1)$ packing \cite{MillsMullin:1992}, and
\item $A(n,2(k-1),k)$, the maximum number of codewords in a binary code of length $n$, 
minimum distance $2(k-1)$, and constant weight $k$ \cite{MacWilliamsSloane:1977}.
\end{itemize}
Despite much effort, the exact value of $T(n,\F(k),1)$ is known for all $n$ only when $k=3$
\cite{Schonheim:1966,Spencer:1968} and $k=4$ \cite{Brouwer:1979}.
Even for $k=5$, there are an infinite number of $n$ for which $T(n,\F(5),1)$ is not yet determined.
In this paper, we determine $T(n,\F(k),2)$ for all $n$ when $k=3$ and for all
but 11 values of $n$ when $k=4$.

\section{Design-theoretic preliminaries}

Our determination of $T(n,\F(k),2)$, $k\in\{3,4\}$, makes extensive use of combinatorial
designs.
In this section, we review some design-theoretic constructs and review some prior results that
are needed in our\break solution.

For positive integers $i\leq j$, the set $\{i, i+1,\ldots,j\}$ is denoted $[i,j]$. The set
$[1,j]$ is further abbreviated as $[j]$. 
A $k$-graph $(X,\A)$ of order $n$ is a {\rm packing of pairs by $k$-tuples}, or more commonly known
as a 2-$(n,k,1)$ {\em packing} if every 2-subset of $X$ is contained in at most one block
of $\A$. The {\em leave} of $(X,\A)$ is the graph $L=(X,\E)$, where $\E$ consists of all
2-subsets of $X$ that are not contained in any blocks of $\A$. We also say that $(X,\A)$ is a 2-$(n,k,1)$
packing {\em leaving} $L$. Given a graph $G$, the maximum size of a 2-$(n,k,1)$ packing
whose leave contains $G$ is denoted $m(n,k,G)$. Note that the maximum size of a
2-$(n,k,1)$ packing, $D(n,k,2)$, is the quantity $m(n,k,G)$ when $G$ is the empty graph.

\begin{theorem}[Sch\"{o}nheim \cite{Schonheim:1966}, Spencer \cite{Spencer:1968}]
For all $n\geq 0$, we have
\begin{eqnarray*}
D(n,3,2) & = & \begin{cases}
\left\lfloor \frac{n}{3} \left\lfloor \frac{n-1}{2} \right\rfloor \right\rfloor -1&\text{if $n\equiv 5$ {\rm (mod 6)},} \\ 
\left\lfloor \frac{n}{3} \left\lfloor \frac{n-1}{2} \right\rfloor \right\rfloor&\text{otherwise.}
\end{cases}
\end{eqnarray*}
\end{theorem}\unskip

\begin{theorem}[Brouwer \cite{Brouwer:1979}]\label{Bpacking}
For all $n\geq 0$, we have
\begin{eqnarray*}
D(n,4,2) & = & \begin{cases}
\left\lfloor \frac{n}{4} \left\lfloor \frac{n-1}{3} \right\rfloor \right\rfloor -1&\text{if $n\equiv 7$ or
$10$ {\rm (mod 12)} and $n\notin\{10,19\}$,} \\ 
\left\lfloor \frac{n}{4} \left\lfloor \frac{n-1}{3} \right\rfloor \right\rfloor -1&\text{if $n\in\{9,17\}$,} \\ 
\left\lfloor \frac{n}{4} \left\lfloor \frac{n-1}{3} \right\rfloor \right\rfloor -2&\text{if $n\in\{8,10,11\}$,} \\
\left\lfloor \frac{n}{4} \left\lfloor \frac{n-1}{3} \right\rfloor \right\rfloor -3&\text{if $n=19$,} \\
\left\lfloor \frac{n}{4} \left\lfloor \frac{n-1}{3} \right\rfloor \right\rfloor&\text{otherwise.}
\end{cases}
\end{eqnarray*}
\end{theorem}\unskip

A {\em pairwise balanced design} (PBD) is a set system $(X,\A)$ such that every 2-subset of $X$
is contained in exactly one block of $\A$. If a PBD is of order $n$ and has a set of block sizes $K$,
we denote it by PBD$(n,K)$. If a member $k\in K$ is superscripted with
a ``$\star$'' (written ``$k^\star$''), it means
that the PBD has exactly one block of size $k$. We require the following result on the existence
of PBDs.

\begin{theorem}[Fort and Hedlund \cite{FortHedlund:1958}]\label{FH}
There exists a ${\rm PBD}(n,\{3,5^\star\})$ if and only if $n\equiv 5$ {\rm (mod 6)}.
\end{theorem}

\begin{theorem}[Rees and Stinson \cite{ReesStinson:1989b}]\label{RS}
There exists a ${\rm PBD}(n,\{4,f^\star\})$ if and only if $n\geq 3f+1$, and
\begin{enumerate}
\item[{\rm (i)}] $n\equiv 1$ or $4$ {\rm (mod 12)} and $f\equiv 1$ or $4$ {\rm (mod 12)} or
\item[{\rm (ii)}] $n\equiv 7$ or $10$ {\rm (mod 12)} and $f\equiv 7$ or $10$ {\rm (mod 12)}.
\end{enumerate}
\end{theorem}

\looseness=-1Let $(X,\A)$ be a set system, and let $\G=\{G_1,\ldots,G_s\}$ be a partition of $X$ into
subsets, called {\em groups}. The triple $(X,\G,\A)$ is a {\em group divisible design} (GDD)
when every 2-subset of $X$ not contained in a group appears in exactly one block, and
$|A\cap G|\leq 1$ for all $A\in\A$ and $G\in\G$. We denote a GDD $(X,\G,\A)$
by $K$-GDD if $K$ is a set of block sizes for $(X,\A)$. The {\em type} of a GDD $(X,\G,\A)$
is the multiset $[|G| : G\in\G]$. When more convenient, we use the exponentiation notation
to describe the type of a GDD: A GDD of type $g_1^{t_1}\ldots g_s^{t_s}$ is a GDD where
there are exactly $t_i$ groups of size $g_i$, $i\in[s]$. The following results on the existence
of $\{4\}$-GDDs are useful. 

\begin{theorem}[Hanani \cite{Hanani:1975}]\label{H}
There exists a $\{3\}$-{\rm GDD} of type $g^t$ if and only if $t\geq 3$,
$g^2{t\choose 2}\equiv 0$ {\rm (mod 3)}, and $g(t-1)\equiv 0$ {\rm (mod 2)}. 
\end{theorem}

\begin{theorem}[Brouwer, Schrijver, and Hanani \cite{Brouweretal:1977}]\label{BSH}
There exists a $\{4\}$-{\rm GDD} of type $g^t$ if and only if $t\geq 4$ and
\begin{enumerate}
\item[{\rm (i)}] $g\equiv 1$ or $5$ {\rm (mod 6)} and $t\equiv 1$ or $4$ {\rm (mod 12)} or
\item[{\rm (ii)}] $g\equiv 2$ or $4$ {\rm (mod 6)} and $t\equiv 1$ {\rm (mod 3)} or
\item[{\rm (iii)}] $g\equiv 3$ {\rm (mod 6)} and $t\equiv 0$ or $1$ {\rm (mod 4)} or
\item[{\rm (iv)}] $g\equiv 0$ {\rm (mod 6)},
\end{enumerate}
with the two exceptions of types $2^4$ and $6^4$, for which $\{4\}$-{\rm GDD}s do not exist. 
\end{theorem}

\begin{theorem}[Brouwer \cite{Brouwer:1979}]\label{B}
A $\{4\}$-{\rm GDD} of type $2^u5^1$ exists if and only if $u=0$, or
$u\equiv 0$ {\rm (mod 3)} and $u\geq 9$. 
\end{theorem}

\begin{theorem}[see \cite{KreherStinson:1997}]\label{KS}
There exists a $\{4\}$-{\rm GDD} of type $3^tu^1$ if and only if $t=0$, or $t\geq (2u+3)/3$ and
\begin{enumerate}
\item[{\rm (i)}] $t\equiv 0$ or $1$ {\rm (mod 4)} and $u\equiv 0$ or $6$ {\rm (mod 12)} or
\item[{\rm (ii)}] $t\equiv 0$ or $3$ {\rm (mod 4)} and $u\equiv 3$ or $9$ {\rm (mod 12)}.
\end{enumerate}
\end{theorem}

\begin{theorem}[Ge and Ling \cite{GeLing:2004}]\label{GL0}
There exists a $\{4\}$-{\rm GDD} of type $2^tu^1$ for $t=0$ and
for each $t\geq 6$ with $t\equiv 0$ {\rm (mod 3)},
$u\equiv 2$ {\rm (mod 3)}, and $2\leq u\leq t-1$, except for $(t,u)=(6,5)$ and except possibly for
$(t,u)\in\{(21,17)$, $(33,23)$, $(33,29)$, $(39,35)$, $(57,44)\}$. 
\end{theorem}

\begin{theorem}[Ge and Ling \cite{GeLing:2004}]\label{GL}
There exists a $\{4\}$-{\rm GDD} of type $12^tu^1$ for $t=0$ and
for each $t\geq 4$ and $u\equiv 0$ {\rm (mod 4)}
such that $0\leq u\leq 6(t-1)$. 
\end{theorem}

An {\em incomplete transversal design of group size} $n$, {\em block size} $k$, and {\em hole size}
$h$ is a quadruple $(X,\G,H,\A)$ such that
\begin{enumerate}
\item[{\rm (i)}] $(X,\A)$ is a $k$-graph of order $nk$;
\item[{\rm (ii)}] $\G$ is a partition of $X$ into $k$ subsets (called {\em groups}), each of cardinality $n$;
\item[{\rm (iii)}] $H\subseteq X$, with the property that, for each $G\in\G$, $|G\cap H|=h$; and
\item[{\rm (iv)}] every 2-subset of $X$ is
	\begin{itemize}
		\item contained in the {\em hole} $H$ and not contained in any blocks or 
		\item contained in a group and not contained in any blocks or 
		\item contained in neither a hole nor a group and contained in exactly one block of $\A$.
	\end{itemize}
\end{enumerate}
Such an incomplete transversal design is denoted ${\rm TD}(k,n)-{\rm TD}(k,h)$. 

\begin{theorem}[Heinrich and Zhu \cite{HeinrichZhu:1986}]\label{HZ}
For $n>h>0$, a ${\rm TD}(4,n)-{\rm TD}(4,h)$ exists if and only if $n\geq 3h$ and
$(n,h)\not=(6,1)$.
\end{theorem}

\section{Packings with leaves containing specified graphs}

In this section, we relate the problem of determining $T(n,\F(k),2)$ to that of determining
$m(n,k,G)$ for $G$ isomorphic to $K_4-e$, $K_5-e$, and $2\circ K_4$ (edge-gluing
of two $K_4$'s) when $k\in\{3,4\}$. These graphs are shown in Figures \ref{fig3.1}--\ref{fig3.3}, respectively.

\begin{lemma}
\label{m+2}
There exists a $3$-graph of order $n$ and size $m$ containing
exactly one copy of an element of $\F(3)$ if and only if there exists a
$2$-$(n,3,1)$ packing of size $m-2$ with a leave containing $K_4-e$ as a subgraph.
\end{lemma}

\begin{figure}[tb]
\centerline{\includegraphics[width=0.8in]{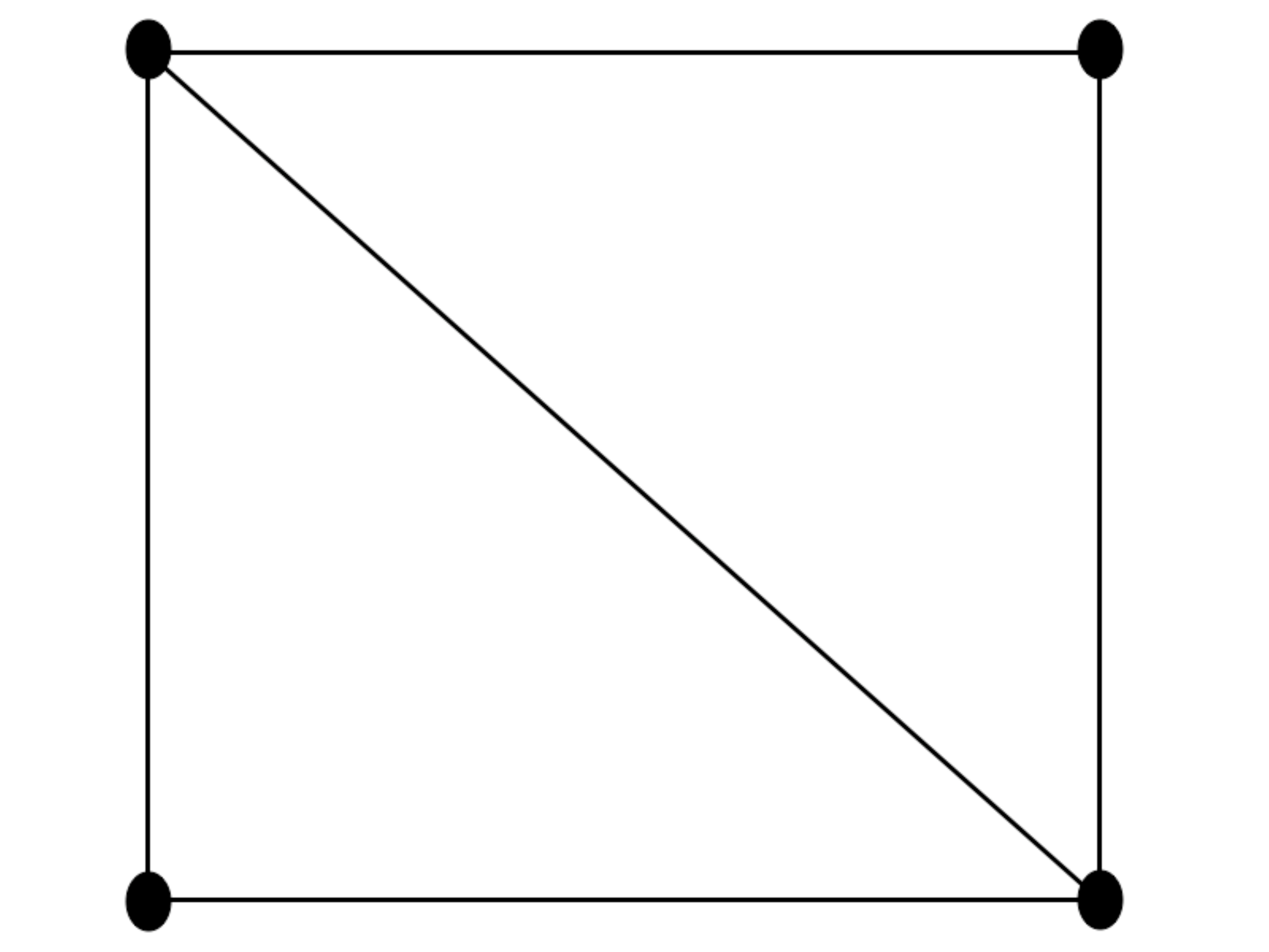}}
\caption{$K_4-e$.}
\label{fig3.1}
\end{figure}

\begin{figure}[tb]
\centerline{\includegraphics[width=1in]{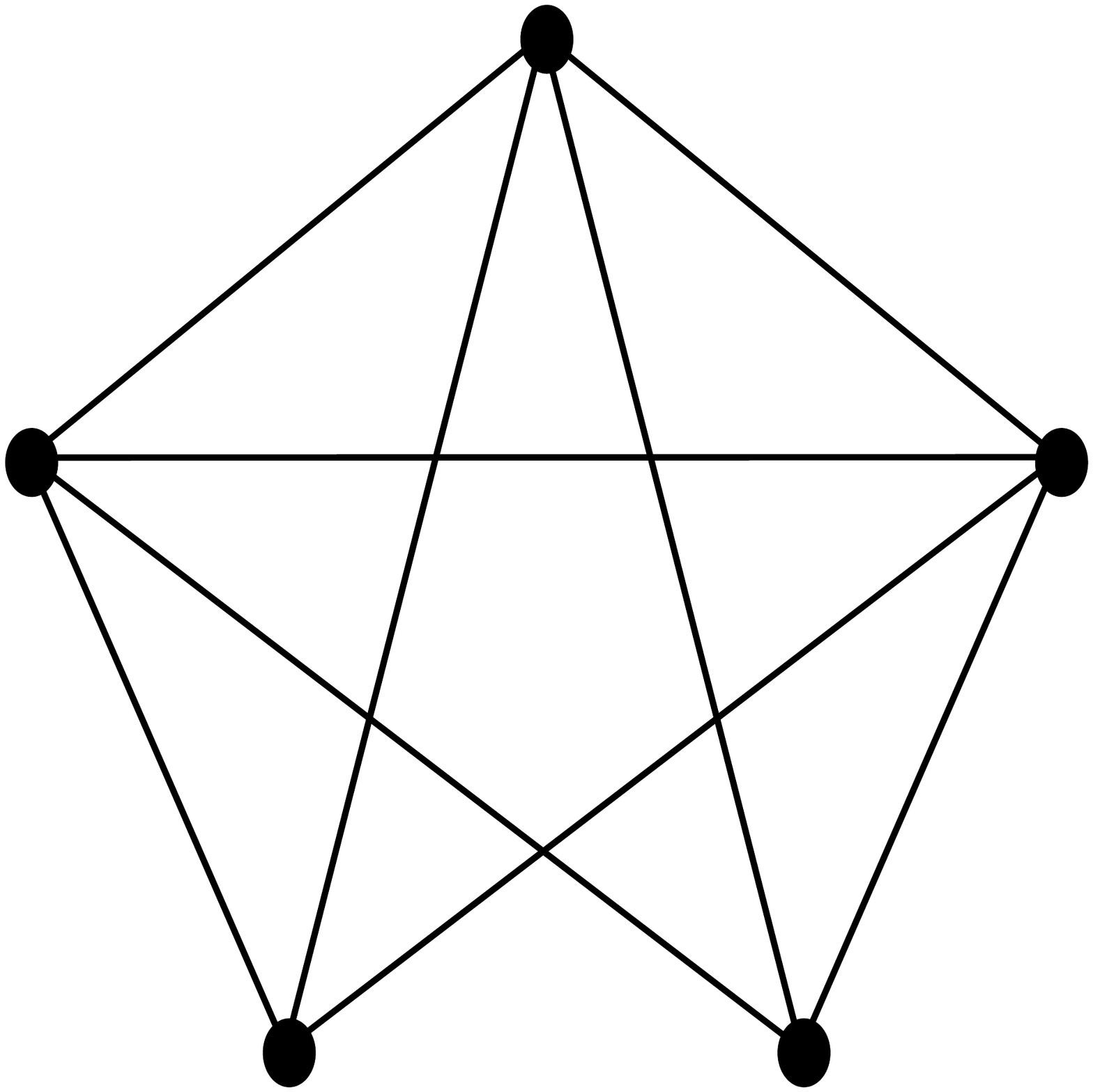}}
\caption{$K_5-e$.}
\label{fig3.2}
\end{figure}

\begin{figure}[tb]
\centerline{\includegraphics[width=1in]{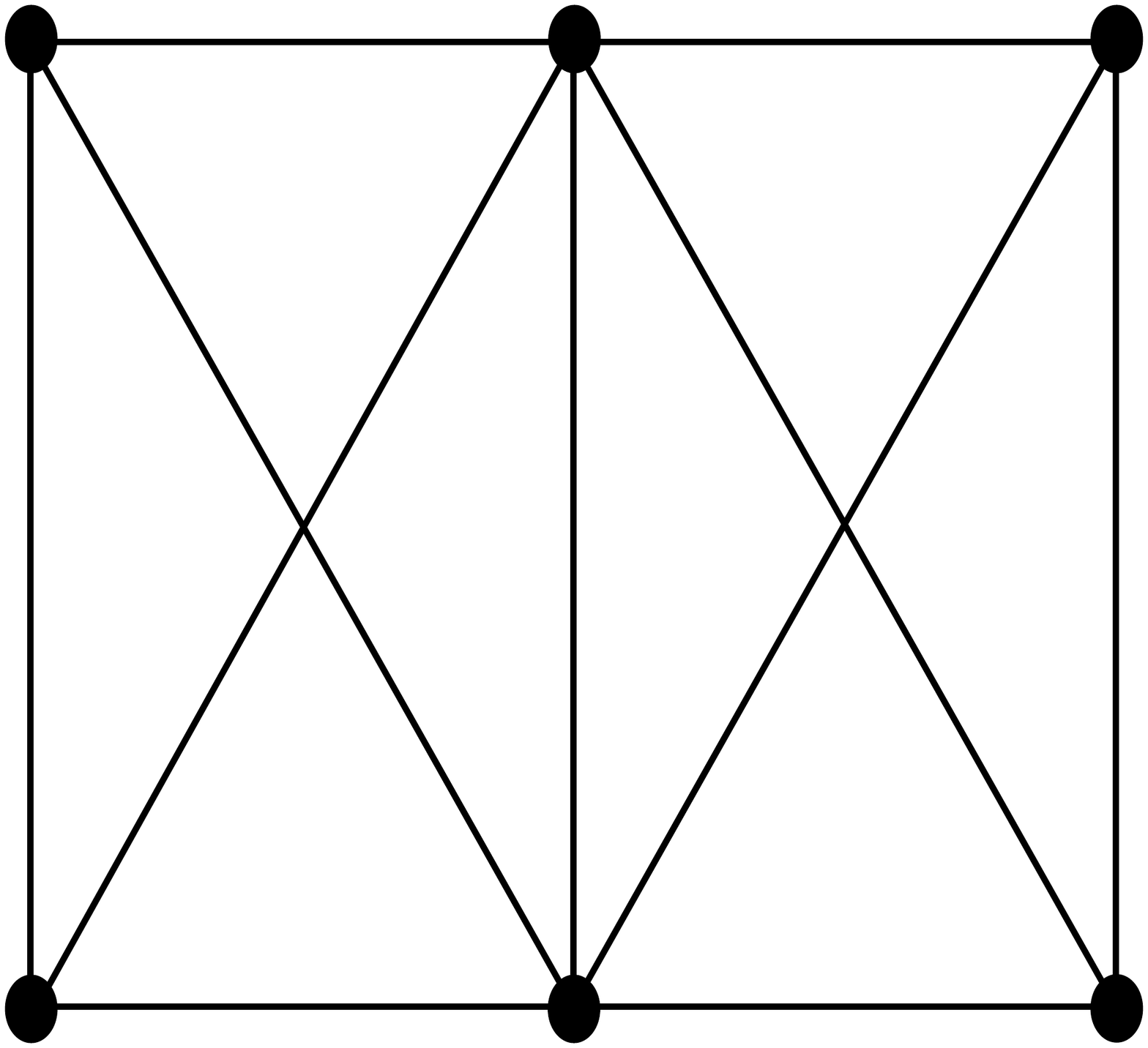}}
\caption{$2\circ K_4$.}
\label{fig3.3}
\end{figure}

\begin{proof}
$\F(3)$ contains only a single 3-graph, $\Lambda(3,2)$.
Let $(X,\A)$ be a $3$-graph of order $n$ and size $m$ containing
exactly one copy of $\Lambda(3,2)$. Then there exist exactly
two blocks $A,B\in\A$, with $|A\cap B|=2$. Let $P=(X,\A\setminus\{A,B\})$. Then $P$ is a
$2$-$(n,3,1)$ packing of size $m-2$ with a leave containing the 2-subsets in $X$ that occurs
in $A$ and $B$, which together form a $K_4-e$. This construction is
reversible.\qquad\end{proof}

\begin{corollary}
The following holds:
\begin{equation*}
T(n,\F(3),2)=\max\{ T(n,\F(3),1), m(n,3,K_4-e)+2\}.
\end{equation*}
\end{corollary}\unskip

\begin{proof}
If a 3-graph contains no two isomorphic copies of $\Lambda(3,2)$, then either it contains no
copies, in which case its maximum size is given by $T(n,\F(3),1)$, or else it contains exactly one
copy, in which case its maximum size is given by $m(n,3,K_4-e)+2$.\qquad\end{proof}

The proofs for the following two lemmas are similar to that for Lemma \ref{m+2} and are thus
omitted. 

\begin{lemma}
\label{2K4}
There exists a $4$-graph of order $n$ and size $m$ containing
exactly one copy of $\Lambda(4,2)$ if and only if there exists a
$2$-$(n,4,1)$ packing of size $m-2$ with a leave containing $2\circ K_4$
as a subgraph.
\end{lemma}

\begin{lemma}
\label{K5-e}
There exists a $4$-graph of order $n$ and size $m$ containing
exactly one copy of $\Lambda(4,3)$ if and only if there exists a
$2$-$(n,4,1)$ packing of size $m-2$ with a leave containing $K_5-e$
as a subgraph.
\end{lemma}

\begin{corollary}
The following holds:
\begin{equation*}
T(n,\F(4),2)=\max\{T(n,\F(4),1), m(n,4,2\circ K_4)+2, m(n,4,K_5-e)+2\}.
\end{equation*}
\end{corollary}\unskip

\begin{proof}
$\F(4)$ contains the graphs $\Lambda(4,2)$ and $\Lambda(4,3)$. So if a 4-graph contains no two
isomorphic copies of an element of $\F(4)$, then either it contains none of them, in which case
its maximum size is given by $T(n,\F(4),1)$, or else it contains exactly one of $\Lambda(4,2)$
or $\Lambda(4,3)$. In the former case, its maximum size is $m(n,4,2\circ K_4)+2$ by Lemma \ref{2K4}, 
and, in the latter case, its maximum size is $m(n,4,K_5-e)$ by\break Lemma \ref{K5-e}.\qquad\end{proof}

\section{\boldmath Determining $T(n,\F(3),2)$}

When $n\equiv 1$ or $3$ {\rm (mod 6)}, a 2-$(n,3,1)$ packing
of size $T(n,\F(3),1)$ has the property that every pair of distinct
points is contained in exactly one
block. Such a 2-$(n,3,1)$ packing is called a
{\em Steiner triple system} of order $n$ and is denoted STS$(n)$.

Let $P=(X,\A)$ be a 2-$(n,3,1)$ packing.
When $n\equiv 1$ or 3 (mod 6), the leave $L=(X,\E)$ of $P$ must satisfy:
\begin{enumerate}
\item[{\rm (i)}] $|\E|\equiv 0$ (mod 3), and
\item[{\rm (ii)}] the degree of every vertex in $L$ is even.
\end{enumerate}
Any $L$ containing $K_4-e$ as a subgraph and satisfying conditions (i) and (ii) above has
at least nine edges. Hence, the maximum size of a $2$-$(n,3,1)$ packing with a leave containing
$K_4-e$ is at most $\frac{1}{3}({n\choose 2}-9)$. We show below that there indeed exists
such a $2$-$(n,3,1)$ packing of size $\frac{1}{3}({n\choose 2}-9)$. 

\begin{lemma}
There exists a $2$-$(n,3,1)$ packing of size
$\frac{1}{3}({n\choose 2}-9)$, with a leave containing $K_4-e$,
for every $n\equiv 1$ or $3$ {\rm (mod 6)}.
\end{lemma}

\begin{proof}
Let $(X,\A)$ be an STS$(n)$. Suppose there exist three blocks in $\A$ of the form
$\{1,2,3\}$, $\{1,4,5\}$, and $\{3,4,a\}$. Then deleting these three blocks
gives a $2$-$(n,3,1)$ packing of size $\frac{1}{3}({n\choose 2}-9)$ with
a leave containing $K_4-e$. Hence, it suffices to show that we can always find such a 3-block
configuration in any STS$(n)$. To see that this is true, pick any two intersecting blocks in an STS$(n)$,
say, $\{1,2,3\}$ and $\{1,4,5\}$. As the third block, take the unique
block containing the 2-subset $\{3,4\}$.\qquad\end{proof}

Next, we consider $n\equiv 5$ (mod 6). In this case, ${n\choose 2}\equiv 1$ (mod 3).
So if the leave of a $2$-$(n,3,1)$ packing contains $K_4-e$, then it must contain at least
seven edges. Therefore, such a packing can have at most $\frac{1}{3}({n\choose 2}-7)$
blocks. We show below that this upper bound can be met using pairwise balanced designs. 

\begin{lemma}
There exists a $2$-$(n,3,1)$ packing of size
$\frac{1}{3}({n\choose 2}-7)$, with a leave containing $K_4-e$,
for every $n\equiv 5$ {\rm (mod 6)}. 
\end{lemma}

\begin{proof}
Let $(X,\A)$ be a PBD$(n,\{3,5^\star\})$ with $[5]$ as the block of size five.
The existence of such a PBD is provided by Theorem \ref{FH}.
Deleting the block of size five from this PBD
and adding the block $\{1,2,3\}$ yield the desired 2-$(n,3,1)$\break packing.\qquad\end{proof}

For $n\equiv 0$, 2, or 4 (mod 6), every vertex in the leave $L$ of a $2$-$(n,3,1)$
packing is of odd degree. If $L$ contains $K_4-e$, then $L$ must have at least four vertices
of degree at least three. The minimum possible number of edges in $L$, if $L$ contains
$K_4-e$, is therefore $n/2+4$. It follows that the number of blocks in a
2-$(n,3,1)$ packing with a leave containing $K_4-e$ is at most
$\left\lfloor\frac{1}{3}({n\choose 2}-\frac{n}{2}-4)\right\rfloor$. 

\begin{lemma}
There exists a $2$-$(n,3,1)$ packing of size $\frac{1}{3}({n\choose 2}-\frac{n}{2}-4)$,
with a leave containing $K_4-e$, for every $n\equiv 4$ {\rm (mod 6)}. 
\end{lemma}

\begin{proof}
Let $(X,\A)$ be a PBD$(n+1,\{3,5^\star\})$ which exists by
Theorem \ref{FH}. Let $x$ be a point contained in the block of size five.
Then $(X\setminus\{x\},\B)$, where
\begin{eqnarray*}
\B & = & \{A\in\A:\text{$x\not\in A$ and $|A|=3$}\}
\end{eqnarray*}
is the desired $2$-$(n,3,1)$ packing.\qquad\end{proof}

\begin{lemma}
There exists a $2$-$(n,3,1)$ packing of size $\frac{1}{3}({n\choose 2}-\frac{n}{2}-6)$,
with a leave containing $K_4-e$, for every $n\equiv 0$ or $2$ {\rm (mod 6)}. 
\end{lemma}

\begin{proof}
Consider a $\{3\}$-GDD of type $2^{n/2}$, which exists whenever $n\equiv 0$ or 2 (mod 6)
by Theorem \ref{H}. Without loss of generality, we may assume $\{1,2\}$ is a group and
$\{1,3,4\}$ is a block in this GDD. There is a unique block of the form $\{2,3,a\}$. Deleting the
blocks $\{1,3,4\}$ and $\{2,3,a\}$ from this GDD gives a $2$-$(n,3,1)$ packing of size
$\frac{1}{3}({n\choose 2}-\frac{n}{2}-6)$, with a leave containing
$K_4-e$.\qquad\end{proof}

This completes our determination of $m(n,3,K_4-e)$. We summarize our results above
as follows.

\begin{theorem}\label{D3(2,2)}
For all $n\geq 0$, we have $m(n,3,K_4-e)=\frac{1}{3}({n\choose 2}-f(n))$, where
\begin{eqnarray*}
f(n) & = & \begin{cases}
n/2+6&\text{if $n\equiv 0$ or {\rm 2 (mod 6)},} \\
9&\text{if $n\equiv 1$ or {\rm 3 (mod 6)},} \\
n/2+4&\text{if $n\equiv 4$ {\rm (mod 6)},} \\
7&\text{if $n\equiv 5$ {\rm (mod 6)}.}
\end{cases}
\end{eqnarray*}
\end{theorem}\unskip

\section{\boldmath Determining $T(n,\F(4),2)$}

We now determine $T(n,\F(4),2)$.

\subsection{\boldmath The case $n\equiv 1$ or $4\pmod{12}$}

The leave $L=(X,\E)$ of a 2-$(n,4,1)$ packing must satisfy:   
\begin{enumerate}
\item[{\rm (i)}] $|\E|\equiv 0$ {\rm (mod 6)}, and
\item[{\rm (ii)}] every vertex in $L$ has degree $\equiv 0$ {\rm (mod 3)}.
\end{enumerate}
Any leave of $P$ containing $K_5-e$ or $2\circ K_4$ as a subgraph and satisfying conditions
(i) and (ii) above has at least 18 edges. So
$m(n,4,G)\leq \frac{1}{6}({n\choose 2}-18)$ for $G\in\{K_5-e,2\circ K_4\}$.
We show below that this bound can be met with a finite number
of possible exceptions.

The {\em cocktail party graph} CP$(n)$ is the unique $(2n-2)$-regular graph
on $2n$ vertices. We begin with an observation on CP$(4)$ (shown in Figure \ref{figcp4}).

\begin{figure}[b!]
\vspace{-6pt}
\centerline{\includegraphics[width=1.2in]{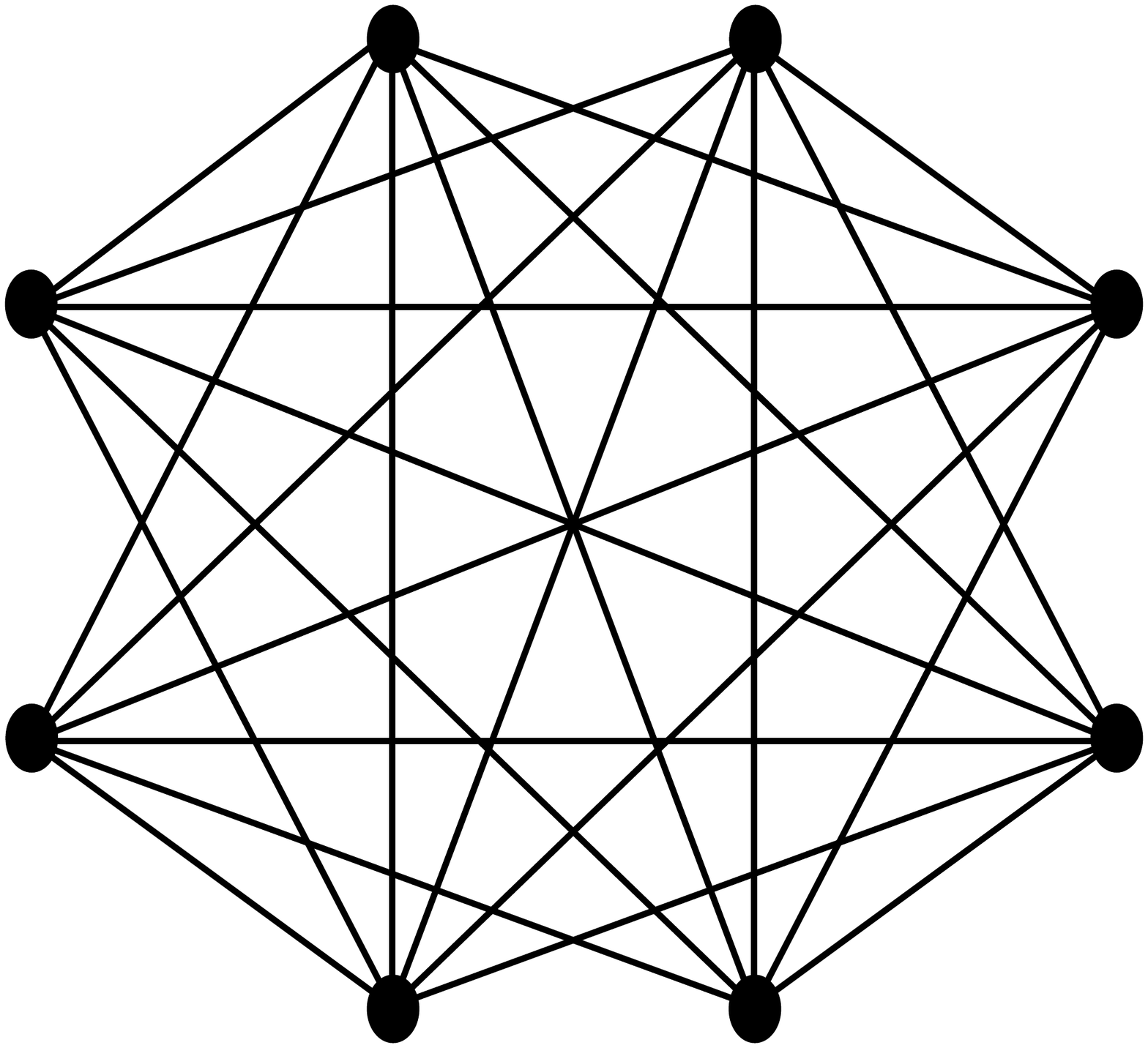}}
\caption{{\rm CP(4)}.}
\label{figcp4}
\end{figure}

\begin{lemma}
\label{CP4K5-e}
{\rm CP}$(4)$ contains an edge-disjoint union of a $K_5-e$ and a $K_4$. 
\end{lemma}

\begin{proof}
Without loss of generality, we may take the vertex set and edge set of the CP$(4)$ as
$[8]$ and $\{A\subset [8]: |A|=2\}\setminus\{\{i,i+4\}: i\in[4]\}$, respectively.
Consider the subsets of edges
$\E_1=\{A\subset\{1,2,3,5,8\} : |A|=2\}\setminus\{\{1,5\}\}$ and
$\E_2=\{A\subset\{2,4,6,7\} : |A|=2\}$. $\E_1$ is the edge set of a $K_5-e$, $\E_2$
is the edge set of a $K_4$, and they are disjoint.\qquad\end{proof}

\begin{lemma}
\label{CP4G2}
{\rm CP$(4)$} contains an edge-disjoint union of a $2\circ K_4$ and a $K_4$. 
\end{lemma}

\begin{proof}
Without loss of generality, we may take the vertex set and edge set of the CP$(4)$ as
$[8]$ and $\{A\subset [8]: |A|=2\}\setminus\{\{i,i+4\}: i\in[4]\}$, respectively.
Consider the subsets of edges
$\E_1=\{A\subset[4] : |A|=2\}\cup(\{A\subset [3,6] : |A|=2\}\setminus\{\{3,4\}\})$ and
$\E_2=\{A\subset\{1,6,7,8\} : |A|=2\}$. $\E_1$ is the edge set of a $2\circ K_4$,  $\E_2$
is the edge set of a $K_4$, and they are disjoint.\qquad\end{proof}

\begin{lemma}
Let $G\in\{K_5-e,2\circ K_4\}$ and $n \equiv 1$ or $4$ {\rm (mod 12)}.
If there exists a $2$-$(n,4,1)$ packing leaving ${\rm CP}(4)$,
then there exists a $2$-$(n,4,1)$
packing of size $\frac{1}{6}({n\choose 2}-18)$ with a leave containing $G$. 
\end{lemma}

\begin{proof}
A 2-$(n,4,1)$ packing whose leave is ${\rm CP}(4)$ has size
$\frac{1}{6}({n\choose 2}-24)$. We have seen from Lemmas \ref{CP4K5-e} 
and  \ref{CP4G2} that we can
add one more block of size four to this packing to give a 2-$(n,4,1)$ packing with a leave
containing $G$.\qquad\end{proof}

In view of the above lemma, we now focus on constructing 2-$(n,4,1)$ packings leaving
CP$(4)$. 

\begin{lemma}
\label{TD}
Let $n\geq 6$.
If there exists
a {\rm PBD}$(n+f,\{4,f^\star\})$, then there exists a $2$-$(4n+f,4,1)$ packing leaving
${\rm CP}(4)$. 
\end{lemma}

\begin{proof}
Take a ${\rm TD}(4,n)-{\rm TD}(4,2)$ $(X,\G,H,\A)$, which exists by Theorem \ref{HZ}, and
for each $G\in\G$, let $(G\cup F,\A_G)$ be a PBD$(n+f,\{4,f^\star\})$, where $F$ is the
block of size $f$ in the PBD. Consider the set system $(Y,\B)$, where
$Y=X\cup F$, and $\B=\A\cup(\cup_{G\in\G} \A_G)$
(note that the block of size $F$ is included only once).
$(Y,\B)$ is a 4-graph of order $4n+f$ having the property that every 2-subset
of $X\cup F$ is contained in exactly one block of $\B$, except for those 2-subsets
$\{a,b\}$, with $a\in G\cap H$ and $b\in G'\cap H$ for distinct $G,G'\in \G$, which are not contained in
any blocks of $\B$. $(Y,\B)$ therefore gives the required 2-$(4n+f,4,1)$ packing leaving
${\rm CP}(4)$.\qquad\end{proof}

\begin{lemma}
Let $n \equiv 1$ or $4$ {\rm (mod 12)}
such that $n \geq 40$ and $n\not\in\{73,76,85\}$. Then there exists a
$2$-$(n,4,1)$
packing leaving ${\rm CP}(4)$. 
\end{lemma}

\begin{proof}
Taking a PBD$(n+f,\{4,f^\star\})$, with $(n,f)\in\{$(9,4), (12,1), (13,0), (15,1), (16,0), (21,4),
(24,1), (25,0), (27,1), (28,0)$\}$, whose existence is provided by Theorem \ref{RS},
and applying Lemma \ref{TD}
give 2-$(n,4,1)$ packings leaving CP$(4)$ for $n\in\{40$, $49$, $52$, $61$, $64$,
$88$, $97$, $100$, $109$, $112\}$.
By Theorem \ref{RS}, there exists a PBD$(n,\{4,40^\star\})$ for all $n\equiv 1$ or $4$ {\rm (mod 12)} and
$n\geq 121$.
Break up the block of size 40 in this PBD
with the blocks of a 2-$(40,4,1)$ packing leaving CP$(4)$
to obtain a 2-$(n,4,1)$ packing leaving CP$(4)$.\qquad\end{proof}

\begin{corollary}
Let $n \equiv 1$ or $4$ {\rm (mod 12)} such that $n \geq 40$ and $n\not\in\{73,76,85\}$.
Then $m(n,4,G)=\frac{1}{6}({n\choose 2}-18)$ for $G\in\{K_5-e,2\circ K_4\}$.
\end{corollary}

\subsection{\boldmath The case $n\equiv 7$ or $10\pmod{12}$}

The leave $L=(X,\E)$ must satisfy:
\begin{enumerate}
\item[{\rm (i)}] $|\E|\equiv 3$ {\rm (mod 6)}, and
\item[{\rm (ii)}] every vertex in $L$ has degree $\equiv 0$ {\rm (mod 3)}.
\end{enumerate}

We first consider the case when $L$ contains $K_5-e$.
Any such $L$ satisfying the conditions (i) and (ii) above
must have at least 15 edges. So $m(n,4,K_5-e)\leq \frac{1}{6}
({n\choose 2}-15)$.

\begin{figure}[tb]
\centerline{\includegraphics[width=1.2in]{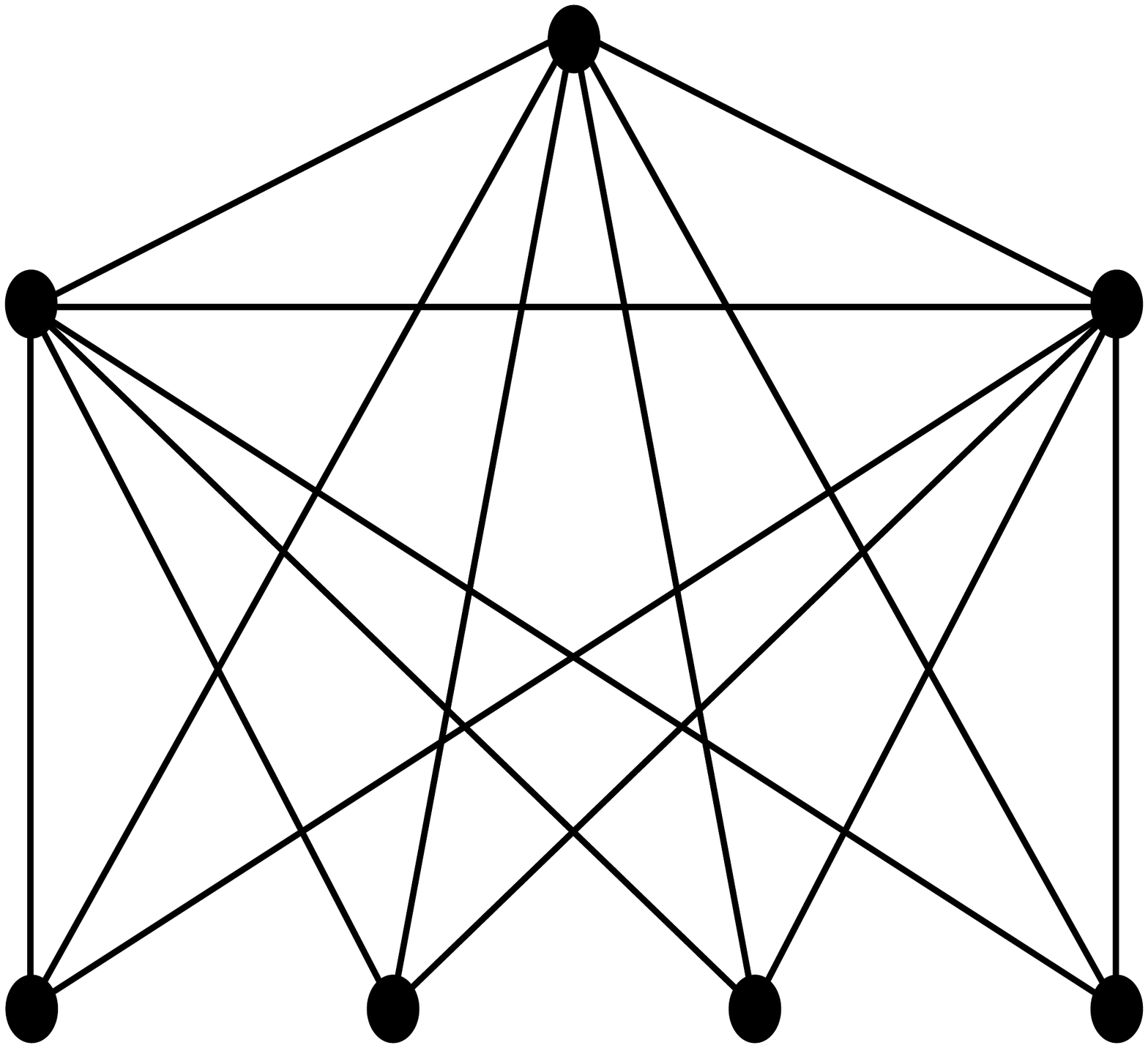}}
\caption{$K_{3,4}+3e$.}
\label{figg3}
\label{fig5.2}
\end{figure}

When $L$ contains $2\circ K_4$, $L$ must also have at least 15 edges.
Suppose $L$ contains $2\circ K_4$ and has 15 edges. Then $L$ must have
at least two vertices, each of degree at least six. Let $a$ be the number of degree three vertices, 
and let $b$ be the number of vertices with degree greater than three in $L$. Then we have
$3a+6b\leq 30$ (counting the edges), $b\geq 2$ (considering the two vertices of degree five
in $2\circ K_4$), and $a+b\geq 7$ (considering the presence of vertices with degree at least six).
These inequalities imply that $2\leq b\leq 3$ and $a+b\leq 8$. So the possible degree sequences
for $L$ are $\D_1=(6,6,6,3,3,3,3)$ and $\D_2=(6,6,3,3,3,3,3,3)$.
Note that we suppress including vertices of degree zero in the degree sequence of $L$.
There is a unique graph with degree sequence $\D_1$, namely, the graph in Figure \ref{fig5.2},
obtained by adding to
$K_{3,4}$ three edges connecting the vertices in the part of the bipartition with three vertices.
This graph does not contain $2\circ K_4$. Hence, $L$ cannot have degree sequence $\D_1$.
If $L$ contains $2\circ K_4$ and has degree sequence $\D_2$, then since $2\circ K_4$
has degree sequence $(5,5,3,3,3,3)$, the two vertices of nonzero degree not in $2\circ K_4$
cannot both be adjacent to the two vertices of degree five in $2\circ K_4$. But this prevents
these two vertices having degree three, a contradiction. Hence $L$ cannot have degree
sequence $\D_2$. It follows that the leave of any 2-$(n,4,1)$ packing containing $2\circ K_4$
must have at least 21 edges, and we have
$m(n,4,2\circ K_4)\leq\frac{1}{6}({n\choose 2}-21)$.

The following shows that these bounds can be met. 

\begin{lemma}
\label{K7}
$K_7$ contains an edge-disjoint union of a $K_5-e$ and a $K_4$. 
\end{lemma}

\begin{proof}
Take the vertex set of the $K_7$ as $[7]$. Consider the subsets of edges
$\E_1=\{A\subset [5] : |A|=2\}\setminus\{\{4,5\}\}$ and $\E_2=\{A\subset [4,7]:|A|=2\}$.
Then $\E_1$ is the edge set of a $K_5-e$,  $\E_2$ is the edge set of a $K_4$, and
they are disjoint.\qquad\end{proof}

\begin{lemma}
Let $n\equiv 7$ or $10$ {\rm (mod 12)} such that $n\geq 7$ and $n\not\in\{10,19\}$. Then
$m(n,4,K_5-e)=\frac{1}{6}({n\choose 2}-15)$. 
\end{lemma}

\begin{proof}
Let $(X,\A)$ be a PBD$(n,\{4,7^\star\})$ with $F$ as the block of size seven, whose existence is provided by Theorem \ref{RS}, and let $B$ be any 4-subset of $F$. Then $(X,(\A\cup\{B\})\setminus\{F\})$ is
a 2-$(n,4,1)$ packing of size $\frac{1}{6}({n\choose 2}-15)$ leaving $K_7-K_4$, which
contains $K_5-e$ by Lemma \ref{K7}.\qquad\end{proof}

\begin{lemma}
Let $n\equiv 7$ or $10$ {\rm (mod 12)} such that $n\geq 7$ and $n\not\in\{10,19\}$. Then
$m(n,4,2\circ K_4)=\frac{1}{6}({n\choose 2}-21)$. 
\end{lemma}

\begin{proof}
Observe that any 2-$(n,4,1)$ packing leaving $K_7$ has size $\frac{1}{6}({n\choose 2}-21)$.
The theorem now follows for $n=7$ trivially and for $n\geq 22$ from the existence of
a PBD$(n,\{4,7^\star\})$ provided by Theorem \ref{RS}.\qquad\end{proof}

\subsection{\boldmath The case $n\equiv 2$, $5$, $8$, or $11\pmod{12}$}

The leave $L=(X,\E)$ must have vertices all of degree $1$ {\rm (mod 3)}. Furthermore,
$|\E|\equiv 1$ {\rm (mod 6)} when $n\equiv 2$ or $11$ {\rm (mod 12)}, and
$|\E|\equiv 4$ {\rm (mod 6)} when $n\equiv 5$ or $8$ {\rm (mod 12)}.

If $L$ contains $K_5-e$, then $L$ must have at least five vertices, each of degree at least four and
the remaining vertices each of degree at least one. Hence, $L$ must have at least
$\frac{1}{2}(n+15)$ edges when $n\equiv 5$ or $11$ {\rm (mod 12)} and at least
$\frac{1}{2}(n+24)$ edges when $n\equiv 2$ or $8$ {\rm (mod 12)}. Consequently,
\begin{eqnarray*}
m(n,4,K_5-e) & \leq & \begin{cases}
\frac{1}{6}({n\choose 2}-\frac{n+15}{2})&\text{if $n\equiv 5$ or $11$ {\rm (mod 12)},}\\[3pt]
\frac{1}{6}({n\choose 2}-\frac{n+24}{2})&\text{if $n\equiv 2$ or $8$ {\rm (mod 12)}.}
\end{cases}
\end{eqnarray*}

If $L$ contains $2\circ K_4$, then $L$ must have at least two vertices, each of degree at least seven,
at least four vertices each of degree at least four, and the rest of the vertices each of degree one. Hence,
$L$ must have at least $\frac{1}{2}(n+24)$ edges when $n\equiv 2$ or $8$ {\rm (mod 12)} 
and at least $\frac{1}{2}(n+27)$ edges when $n\equiv 5$ or $11$ {\rm (mod 12)}.
Consequently,
\begin{eqnarray*}
m(n,4,2\circ K_4) & \leq & \begin{cases}
\frac{1}{6}({n\choose 2}-\frac{n+24}{2})&\text{if $n\equiv 2$ or $8$ {\rm (mod 12)},}\\[3pt]
\frac{1}{6}({n\choose 2}-\frac{n+27}{2})&\text{if $n\equiv 5$ or $11$ {\rm (mod 12)}.}
\end{cases}
\end{eqnarray*}
These bounds can be met with the following constructions.

\subsubsection{\boldmath The value of $m(n,4,K_5-e)$}
\hfill

\begin{lemma}
Let $n\equiv 5$ or $11$ {\rm (mod 12)} such that $n=5$ or $n\geq 23$. Then we have
$m(n,4,K_5-e)=\frac{1}{6}({n\choose 2}-\frac{1}{2}(n+15))$. 
\end{lemma}

\begin{proof}
Let $(X,\G,\A)$ be a $\{4\}$-GDD of type $2^{(n-5)/2}5^1$, which exists by Theorem \ref{B}.
Then $(X,\A)$ is a 2-$(n,4,1)$ packing of size $\frac{1}{6}({n\choose 2}-\frac{1}{2}(n+15))$
with a leave containing $K_5$, and hence $K_5-e$.\qquad\end{proof}

\begin{lemma}
\label{small14}
There exists a $2$-$(14,4,1)$ packing of size $12$ having a leave containing $K_5-e$. 
\end{lemma}

\begin{proof}
Let $(X,\A)$ be a maximum $2$-$(13,4,1)$ packing, which has size 13 by Theorem \ref{Bpacking}.
Let $\infty\not\in X$ and $A\in\A$. Then $(X\cup\{\infty\},\A\setminus\{A\})$ is a
2-$(14,4,1)$ packing of size 12
with a leave containing $K_5$ (whose edges are the 2-subsets of $A\cup\{\infty\}$).\qquad\end{proof}

\begin{lemma}
Let $n\equiv 2$ or $8$ {\rm (mod 12)} such that $n=14$ or $n\geq 44$. Then we have
$m(n,4,K_5-e)=\frac{1}{6}({n\choose 2}-\frac{1}{2}(n+24))$. 
\end{lemma}

\begin{proof}
Let $(X,\G,\A)$ be a $\{4\}$-GDD of type $2^{(n-14)/2}14^1$, which exists by Theorem \ref{GL0}.
Let $G\in\G$ be the group of cardinality 14, and let $(G,\B)$ be a 2-$(14,4,1)$ packing of size
12 having a leave containing $K_5-e$, whose existence is provided by Theorem \ref{small14}.
Then $(X,\A\cup\B)$ is a 2-$(n,4,1)$ packing having a leave containing $K_5-e$. The size of this packing
is $\frac{1}{6}({n\choose 2}-\frac{1}{2}(n-14)-{14\choose 2})+12=
\frac{1}{6}({n\choose 2}-\frac{1}{2}(n+24))$.\qquad\end{proof}

\subsubsection{\boldmath The value of $m(n,4,2\circ K_4)$}
\hfill

\begin{lemma}
\label{4GDD28}
If there exists a $\{4\}$-{\rm GDD} of type $[g_1,\ldots,g_s]$ with $s\geq 3$ 
and a $\{4\}$-{\rm GDD} of type $2^{g_i/2 +1}$ for each $i\in [s]$,
then there exists a $2$-$(n,4,1)$ packing of size $\frac{1}{6}({n\choose 2}-\frac{1}{2}(n+24))$
with a leave contaning $2\circ K_4$, where
$n=2+\sum_{i=1}^s g_i$. 
\end{lemma}

\begin{proof}
Suppose that $(X,\G,\A)$ is a $\{4\}$-GDD of type $[g_1,\ldots,g_s]$, where
$\G=\{G_1,\ldots,G_s\}$ and $|G_i|=g_i$ for $i\in [s]$. Let
$Y=\{\infty_1,\infty_2\}$, where $\infty_1,\infty_2\not\in X$, and let
$(G_i\cup Y,\HH_{G_i},\A_{G_i})$ be a $\{4\}$-GDD of type $2^{g_i/2+1}$ such that
\begin{eqnarray*}
\begin{cases}
Y\in\HH_{G_i}&\text{if $i\in[s-2]$,} \\
\text{$Y$ is contained in a block $A_{G_i}\in\A_{G_i}$}&\text{if $i\in\{s-1,s\}$.}
\end{cases}
\end{eqnarray*}
Construct a 4-graph $(X\cup Y,\B)$ of order $2+\sum_{i=1}^s g_i$, where
\begin{eqnarray*}
\B & = & \A \cup \left(\bigcup_{i=1}^{s} \A_{G_i}\right)
\setminus \{ A_{G_{s-1}}, A_{G_{s}} \}.
\end{eqnarray*}
It is easy to see that $(X\cup Y,\B)$ is a 2-$(2+\sum_{i=1}^s g_i,4,1)$ packing.
Also, the 2-subsets of $A_{G_{s-1}}$ and $A_{G_{s}}$ are not contained in any blocks of $\B$. So
the leave of $(X\cup Y,\B)$ contains $2\circ K_4$ as a subgraph. It remains to compute
the size of $(X\cup Y,\B)$. The 2-subsets of $X\cup Y$ that are not contained in any blocks of $\B$
are precisely the elements of $\HH_{G_i}$ for $i\in [s]$ 
and the 2-subsets of $A_{G_{s-1}}$ and $A_{G_s}$.
Since $Y$ appears precisely $s$ times among these 2-subsets, the total number of
distinct 2-subsets of $X\cup Y$
that are not contained in any blocks of $\B$ is $\sum_{i=1}^s (g_i/2+1)+12-(s-1)=n/2+12$,
where $n=2+\sum_{i=1}^s g_i$. Hence $|\B|=\frac{1}{6}({n\choose 2}-\frac{1}{2}(n+24))$,
as required.\qquad\end{proof}

\begin{lemma}
\label{4GDD511}
If there exists a $\{4\}$-{\rm GDD} of type $[g_1,\ldots,g_s]$ with $s\geq 3$,
a $\{4\}$-{\rm GDD} of type $2^{g_i/2+1}$ for each $i\in [s-1]$,
and a $\{4\}$-{\rm GDD} of type $2^{(g_s-3)/2}5^1$, 
then there exists a $2$-$(n,4,1)$ packing of size $\frac{1}{6}({n\choose 2}-\frac{1}{2}(n+27))$
with a leave containing $2\circ K_4$, where $n=2+\sum_{i=1}^s g_i$. 
\end{lemma}

\begin{proof}
Suppose that $(X,\G,\A)$ is a $\{4\}$-GDD of type $[g_1,\ldots,g_s]$, where
$\G=\{G_1, \ldots, G_s\}$ and $|G_i|=g_i$ for $i\in[s]$.
Let $Y=\{\infty_1,\infty_2\}$, where $\infty_1,\infty_2\not\in X$, and let
$(G_i\cup Y,\HH_{G_i},\A_{G_i})$ be a $\{4\}$-GDD of type $2^{g_i/2+1}$ such that
\begin{eqnarray*}
\begin{cases}
Y\in\HH_{G_i}&\text{if $i\in[s-3]$,} \\
\text{$Y$ is contained in a block $A_{G_i}\in\A_{G_i}$}&\text{if $i\in\{s-2,s-1\}$.}
\end{cases}
\end{eqnarray*}
Further, let $(G_s\cup Y, \HH_{G_s},\A_{G_s})$ be a $\{4\}$-GDD of type
$2^{(g_s-3)/2}5^1$ such that $Y$ is contained in the group $H\in\HH_{G_s}$ of cardinality five.
Now form the 4-graph $(X\cup Y,\B)$ of order $2+\sum_{i=1}^s g_i$, where
\begin{eqnarray*}
\B & = & \A \cup \left(\bigcup_{i=1}^{s} \A_{G_i}\right) \cup \{H\setminus\{\infty_1\}\}
\setminus \{A_{G_{s-2}},A_{G_{s-1}}\}.
\end{eqnarray*}
It is easy to see that $(X\cup Y,\B)$ is a 2-$(2+\sum_{i=1}^s g_i,4,1)$ packing.
Also, the 2-subsets of $A_{G_{s-1}}$ and $A_{G_{s}}$ are not contained in any blocks of $\B$. So
the leave of $(X\cup Y,\B)$ contains $2\circ K_4$ as a subgraph.
It remains to compute
the size of $(X\cup Y,\B)$. The 2-subsets of $X\cup Y$ that are not contained in any blocks of $\B$
are precisely the 2-subsets of $A_{G_{s-2}}$ and $A_{G_{s-1}}$
and the 2-subsets of elements of $\HH_{G_i}$ for $i\in [s]$, except for
the 2-subsets of $H\setminus\{\infty_1\}$.
Since $Y$ appears precisely $s$ times among these 2-subsets, the total number of
distinct 2-subsets of $X\cup Y$
that are not contained in any blocks of $\B$ is $\sum_{i=1}^{s-1} (g_i/2+1)+(g_s-3)/2+(10-6)-1+12-(s-1)
=\frac{1}{2}(n+27)$,
where $n=2+\sum_{i=1}^s g_i$. Hence $|\B|=\frac{1}{6}({n\choose 2}-\frac{1}{2}(n+27))$,
as required.\qquad\end{proof}

\begin{corollary}
For  all $n\equiv 2$ {\rm (mod 12)}, $n\geq 50$,
there exists a $2$-$(n,4,1)$ packing of size $\frac{1}{6}({n\choose 2}-\frac{1}{2}(n+24))$ with a 
leave containing $2\circ K_4$. 
\end{corollary}

\begin{proof}
Apply Lemma \ref{4GDD28} with
$\{4\}$-GDDs of type $12^{(n-2)/12}$ and type $2^7$, which exist by Theorem \ref{BSH}.\qquad\end{proof}

\begin{corollary}
For $n=29$ and for all $n\equiv 5$ {\rm (mod 12)}, $n\geq 101$,
there exists a $2$-$(n,4,1)$ packing of size $\frac{1}{6}({n\choose 2}-\frac{1}{2}(n+27))$
with  a leave containing $2\circ K_4$. 
\end{corollary}

\begin{proof}
Apply Lemma \ref{4GDD511} with
$\{4\}$-GDDs of type $12^{(n-29)/12}27^1$, which exists by Theorem \ref{GL},
$\{4\}$-GDDs of type $2^7$, which exists by Theorem \ref{BSH}, and
$\{4\}$-GDDs of type $2^{12}5^1$, which exists by Theorem \ref{B}.\qquad\end{proof}

\begin{corollary}
For $n=20$ and for all $n\equiv 8$ {\rm (mod 12)}, $n\geq 68$,
there exists a $2$-$(n,4,1)$ packing of size $\frac{1}{6}({n\choose 2}-\frac{1}{2}(n+24))$ with a 
leave containing $2\circ K_4$. 
\end{corollary}

\begin{proof}
Apply Lemma \ref{4GDD28} with
$\{4\}$-GDDs of type $12^{(n-20)/12}18^1$, which exists\break by Theorem \ref{GL}, and
$\{4\}$-GDDs of types $2^7$ and $2^{10}$, which exists by\break Theorem \ref{BSH}.\qquad\end{proof}

\begin{corollary}
For $n=23$ and for all $n\equiv 11$ {\rm (mod 12)}, $n\geq 83$,
there exists a $2$-$(n,4,1)$ packing of size $\frac{1}{6}({n\choose 2}-\frac{1}{2}(n+27))$
with a leave containing $2\circ K_4$. 
\end{corollary}

\begin{proof}
Apply Lemma \ref{4GDD511} with
$\{4\}$-GDDs of type $12^{(n-23)/12}21^1$, which exists by Theorem \ref{GL},
$\{4\}$-GDDs of type $2^7$, which exists by Theorem \ref{BSH}, and
$\{4\}$-GDDs of type $2^95^1$, which exists by Theorem \ref{B}.\qquad\end{proof}

\subsection{\boldmath The case $n\equiv 0$, $3$, $6$, or $9\pmod{12}$}

The leave $L=(X,\E)$ must have vertices all of degree $2$ {\rm (mod 3)}.
Furthermore, $|\E|\equiv 0$ {\rm (mod 6)} when $n\equiv 0$ or $9$ {\rm (mod 12)}, and 
$|\E|\equiv 3$ {\rm (mod 6)} when $n\equiv 3$ or $6$ {\rm (mod 12)}.

If $L$ contains $K_5-e$ or $2\circ K_4$,
then $L$ must have at least six vertices each of degree at least five 
and the remaining vertices each of degree at least two. Hence, $L$ must have at least
$n+9$ edges when $n\equiv 6$ or $9$ {\rm (mod 12)} and at least
$n+12$ edges when $n\equiv 0$ or $3$ {\rm (mod 12)}. Consequently, for $G\in\{K_5-e,2\circ K_4\}$,
we have
\begin{eqnarray*}
m(n,4,G) & \leq & \begin{cases}
\frac{1}{6}({n\choose 2}-(n+9))&\text{if $n\equiv 6$ or $9$ {\rm (mod 12)},}\\[5pt]
\frac{1}{6}({n\choose 2}-(n+12))&\text{if $n\equiv 0$ or $3$ {\rm (mod 12)}.}
\end{cases}
\end{eqnarray*}
These bounds can again be met with the following constructions. 

\begin{lemma}
For $n=6$ and for all $n\equiv 6$ or $9$ {\rm (mod 12)}, $n\geq 21$
there exists a $2$-$(n,4,1)$ packing of size $\frac{1}{6}({n\choose 2}-(n+9))$
with a leave containing $G$, where $G\in\{K_5-e,2\circ K_4\}$. 
\end{lemma}

\begin{proof}
Let $(X,\G,\A)$ be a $\{4\}$-GDD of type $3^{(n-6)/3}6^1$,
which exists by Theorem \ref{KS}. Then $(X,\A)$ is a 2-$(n,4,1)$
packing with a leave containing $K_6$, 
and hence $K_5-e$ and $2\circ K_4$. The size of $(X,\A)$ is easily verified:
$|\A|=\frac{1}{6}({n\choose 2}-\frac{n-6}{3}{3\choose 2}-{6\choose 2})=
\frac{1}{6}({n\choose 2}-(n+9))$.\qquad\end{proof}

\begin{lemma}
\label{15}
There exists a $2$-$(15,4,1)$ packing of size $13$ with a leave containing $G$,
where $G\in\{K_5-e,2\circ K_4\}$. 
\end{lemma}

\textit{Proof.}
The 13 blocks of a $2$-$(15,4,1)$ packing with a leave containing $K_5-e$ are\
\begin{equation*}
\begin{array}{ccccc}
$\{2,6,13,14\}$, &
$\{3,6,9,10\}$, &
$\{4,7,9,13\}$, &
$\{4,5,6,12\}$, &
$\{1,6,11,15\}$, \\ [3pt]
$\{3,7,11,14\}$, &
$\{2,7,8,15\}$, &
$\{1,8,9,14\}$, &
$\{3,12,13,15\}$, &
$\{2,9,11,12\}$, \\[3pt]
$\{1,7,10,12\}$, &
$\{5,10,14,15\}$, &
$\{5,8,11,13\}$.
\end{array}
\end{equation*}

The 13 blocks of a $2$-$(15,4,1)$ packing with a leave containing $2\circ K_4$ are
\begin{equation*}
\begin{array}{ccccc}
$\{1,8,12,13\}$, &
$\{6,8,11,14\}$, &
$\{4,6,9,15\}$, &
$\{3,7,8,9\}$, &
$\{2,8,10,15\}$, \\ [3pt]
$\{2,9,13,14\}$,  &
$\{4,5,7,14\}$, &
$\{1,6,7,10\}$, &
$\{1,5,11,15\}$, &
$\{2,7,11,12\}$, \\ [3pt]
$\{4,10,11,13\}$, &
$\{3,12,14,15\}$, &
$\{5,9,10,12\}$.\qquad\endproof
\end{array}
\end{equation*}

\begin{lemma}
For all $n\equiv 0$ or $3$ {\rm (mod 12)}, $n\geq 48$,
there exists a $2$-$(n,4,1)$ packing of size $\frac{1}{6}({n\choose 2}-(n+12))$
with a leave containing $G$, where $G\in\{K_5-e,2\circ K_4\}$. 
\end{lemma}

\begin{proof}
Let $(X,\G,\A)$ be a $\{4\}$-GDD of type $3^{(n-15)/3}15^1$,
which exists by Theorem \ref{KS}. Let $Y$ be the group of
cardinality 15 in $\G$ and $(Y,\B)$ be a 2-$(15,4,1)$ packing of size 13 with a leave containing $G$,
which exists by Lemma \ref{15}. Then $(X,\A\cup\B)$ is a  2-$(n,4,1)$ packing with a 
leave containing $G$.
The size of $(X,\A\cup\B)$ is easily verified:
$|\A\cup\B|=\frac{1}{6}({n\choose 2}-\frac{n-12}{3}{3\choose 2}-2{6\choose 2})+13=
\frac{1}{6}({n\choose 2}-(n+12))$.\qquad\end{proof}

\subsection{Remaining small orders}

The values of $n$ for which $m(n, 4, K_5-e)$ and $m(n,4,2\circ K_4)$ remain undetermined
are as follows:\vspace{10pt}

\begin{center}
{\footnotesize \begin{tabular}{| c | c c c c c c c c c c c c c |}
\hline
& \multicolumn{13}{c|}{Unsettled $n$} \\
\hline
$m(n,4,K_5-e)$ & 8 & 9 & 10 & 11 & 12 & 13 & 16 & 17 & 18 & 19 & 20 & 24 & 25 \\
& 26 & 27 & 28 & 32 & 36 & 37 & 38 & 39 & 73 & 76 & 85 & &   \\
\hline
$m(n,4,2\circ K_4)$ & 8 & 9 & 10 & 11 & 12 & 13 & 14 & 16 & 17 & 18 & 19 & 24 & 25 \\
& 26 & 27 & 28 & 32 & 35 & 36 & 37 & 38 & 39 & 41 & 44 & 47 & 53 \\
& 56 & 59 & 65 & 71 & 73 & 76 & 77 & 85 & 89 & & & &  \\
\hline
\end{tabular}}
\end{center}\vspace{10pt}

For $n=19$, we have the following tighter upper bound.

\begin{lemma}
For $G\in\{K_5-e,2\circ K_4\}$, we have $m(19,4,G) \leq 24$. 
\end{lemma}

\begin{proof}
Suppose we have a $2$-$(19,4,1)$ packing of size 25 with a  leave containing $G$, and then
we can add a $K_4$ in $G$ to this packing, giving a $2$-$(19,4,1)$ packing of size 26. This
is a contradiction, since $D(19,4,2)=25$.\qquad\end{proof}

For values of $n<16$, it is possible to determine $m(n,4,G)$, $G\in\{K_5-e,2\circ K_4\}$, 
via exhaustive search. Let $H$ be a specific subgraph of $K_n$ isomorphic to $G$. We form a graph
$\Gamma_n$ whose vertex set is the set of all $K_4$'s of $K_n-H$, and two vertices in $\Gamma_n$
are adjacent if and only if the corresponding $K_4$'s are edge-disjoint. Then
$m(n,4,G)$ is equal to the size of a maximum clique in $\Gamma_n$. We used {\sf Cliquer},
an implementation of \"{O}sterg{\aa}rd's exact algorithm for maximum cliques \cite{Ostergard:2002},
to determine the size of maximum cliques in $\Gamma_n$, for $n\leq 15$.

When $n\geq 16$, it is infeasible to use {\sf Cliquer}, so we resort to a stochastic local search
heuristic to construct packings of the required size directly.
The results of our computation are summarized in Table \ref{comput}, 
while the blocks of the actual packings are listed in Appendices A and B.

\begin{table}[tb]
\caption{Values of $m(n,4,K_5-e)$ and $m(n,4,2\circ K_4)$ for some small values of $n$. A blank entry
indicates an unknown value.}\label{comput}
\centering
\footnotesize
\begin{tabular}{| c | c c c c c c c c c c c c c |}
\hline
& \multicolumn{13}{c|}{$n$} \\
\hline
$n$ &			8 & 9 & 10 & 11 & 12 & 13 & 16 & 17 & 18 & 19 & 20 & 24 & 25 \\
$m(n,4,K_5-e)$ & 	1 & 2 & 3   & 4    & 6   & 9   &       &       & 21 & 24 & 28 & 40 &       \\  
\hline
$n$ & 26 & 27 & 28 & 32 & 36 & 37 & 38 & 39 & 73 & 76 & 85 & &  \\
$m(n,4,K_5-e)$ & 50 & 52 &      &  & 97 & & & & & & & &  \\
\hline
\hline
$n$ & 				8 & 9 & 10 & 11 & 12 & 13 & 14 & 16 & 17 & 18 & 19 & 24 & 25 \\
$m(n,4,2\circ K_4)$ & 	1 & 2 & 3   & 4    & 6   & 9   & 11 &       &      & 21 &  24 & 40 &       \\
\hline
$n$ & 26 & 27 & 28 & 32 & 35 & 36 & 37 & 38 & 39 & 41 & 44 & 47 & 53 \\
$m(n,4,2\circ K_4)$ &      & 52 & & & & & & & & & & &  \\
\hline
$n$ & 56 & 59 & 65 & 71 & 73 & 76 & 77 & 85 & 89 & & & &  \\
$m(n,4,2\circ K_4)$ & & & & & & & & & & & & & \\
\hline
\end{tabular}
\vspace{-6pt}
\end{table}

\subsection{Piecing things together}

The results in 
previous subsections can be summarized as follows.

\begin{theorem}\label{exK5-e}
For all $n\geq 5$, we have $m(n,4,K_5-e)=\frac{1}{6}({n\choose 2}-f(n))$, where
\begin{eqnarray*}
f(n) = \begin{cases}
18&\text{if $n\equiv 1$ or $4$ {\rm (mod 12)}, $n\not=13$,} \\
15&\text{if $n\equiv 7$ or $10$ {\rm (mod 12)}, $n\not\in\{10,19\}$,} \\
(n+24)/2&\text{if $n\equiv 2$ or $8$ {\rm (mod 12)}, $n\not=8$,} \\
(n+15)/2&\text{if $n\equiv 5$ or $11$ {\rm (mod 12)}, $n\not=11$,} \\
n+9&\text{if $n\equiv 6$ or $9$ {\rm (mod 12)}, $n\not=9$,} \\
n+12&\text{if $n\equiv 0$ or $3$ {\rm (mod 12)}, $n\not=12$,} \\
22&\text{if $n=8$,} \\
24&\text{if $n=9$,} \\
27&\text{if $n=10$,} \\
31&\text{if $n=11$,} \\
30&\text{if $n=12$,} \\
24&\text{if $n=13$,} \\
27&\text{if $n=19$,}
\end{cases}
\end{eqnarray*}
except possibly for
$n\in\{16, 17, 25, 28, 32, 37, 38, 39, 73, 76, 85\}$.
\end{theorem}\enlargethispage{6pt}

\begin{theorem}\label{ex1D422}
For all $n\geq 6$, we have $m(n,4,2\circ K_4)=\frac{1}{6}({n\choose 2}-f(n))$, where
\begin{eqnarray*}
f(n) = \begin{cases}
18&\text{if $n\equiv 1$ or $4$ {\rm (mod 12)}, $n\not=13$,} \\
21&\text{if $n\equiv 7$ or $10$ {\rm (mod 12)}, $n\not\in\{10,19\}$,} \\
(n+24)/2&\text{if $n\equiv 2$ or $8$ {\rm (mod 12)}, $n\not\in\{8,14\}$,} \\
(n+27)/2&\text{if $n\equiv 5$ or $11$ {\rm (mod 12)}, $n\not=11$,} \\
n+9&\text{if $n\equiv 6$ or $9$ {\rm (mod 12)}, $n\not=9$,} \\
n+12&\text{if $n\equiv 0$ or $3$ {\rm (mod 12)}, $n\not=12$,} \\
22&\text{if $n=8$,} \\
24&\text{if $n=9$,} \\
27&\text{if $n=10$,} \\
31&\text{if $n=11$,} \\
30&\text{if $n=12$,} \\
24&\text{if $n=13$,} \\
25&\text{if $n=14$,} \\
27&\text{if $n=19$,}
\end{cases}
\end{eqnarray*}
except possibly for
{\rm 
$n\in\{16$, 17, 25, 26, 28, 32, 35, 36, 37, 38, 39, 41, 44, 47, 53, 56, 59, 65, 71, 73, 76, 77, 85, $89\}$.
}
\end{theorem}

\section{Conclusion}

Theorems \ref{D3(2,2)}, \ref{exK5-e}, and \ref{ex1D422} can be expressed more succinctly in terms
of $D(n,3,2)$ and $D(n,4,2)$ as follows. 

\begin{theorem}
For all $n\geq 4$,
\begin{eqnarray*}
m(n,3,K_4-e)+2 = \begin{cases}
D(n,3,2)&\text{if $n\equiv 0$, $2$, or $5$ {\rm (mod 6)}}, \\
D(n,3,2)-1&\text{if $n\equiv 1$, $3$, or $4$ {\rm (mod 6)}}.
\end{cases}
\end{eqnarray*}
\end{theorem}

\begin{theorem}
For all $n\geq 5$,
\begin{eqnarray*}
m(n,4,K_5-e)+2 = \begin{cases}
D(n,4,2)+1&\text{if $n\equiv 5$, $6$, $7$, $9$, $10$, or $11$ {\rm (mod 12)}, } \\
& \text{$n\not\in\{9,10,11\}$,} \\
D(n,4,2)&\text{if $n\equiv 0$, $2$, $3$, or $8$ {\rm (mod 12)}, $n\not\in\{8,12\}$,} \\
D(n,4,2)-1&\text{if $n\equiv 1$ or $4$ {\rm (mod 12)}, $n\not=13$,} \\
n-5&\text{if $n\in\{8,9,10,11\}$,} \\
8&\text{if $n=12$,} \\
11&\text{if $n=13$,} \\
\end{cases}
\end{eqnarray*}
except possibly for
$n\in\{16, 17, 25, 28, 32, 37, 38, 39, 73, 76, 85\}$.
\end{theorem}

\begin{theorem}
For all $n\geq 6$,
\begin{eqnarray*}
m(n,4,2\circ K_4)+2=\begin{cases}
D(n,4,2)+1&\text{if $n\equiv 6$ or $9$ {\rm (mod 12)}, $n\not=9$,} \\
D(n,4,2)&\text{if $n\equiv 0$, $2$, $3$, $5$, $7$, $8$, $10$, or $11$ {\rm (mod 12)},} \\
& \text{$n\not\in\{8,10,11,12,14\}$,} \\
D(n,4,2)-1&\text{if $n\equiv 1$ or $4$ {\rm (mod 12)}, $n\not=13$,} \\
n-5&\text{if $n\in\{8,9,10,11\}$,} \\
8&\text{if $n=12$,} \\
11&\text{if $n=13$,} \\
13&\text{if $n=14$,} \\
\end{cases}
\end{eqnarray*}
except possibly for
{\rm 
$n\in\{16$, 17, 25, 26, 28, 32, 35, 36, 37, 38, 39, 41, 44, 47, 53, 56, 59, 65, 71, 73, 76, 77, 85, $89\}$.} 
\end{theorem}

These have the following consequences.

\begin{corollary}
For all $n\geq 4$, $T(n,\F(3),2)=D(n,3,2)$. 
\end{corollary}

\begin{corollary}
For all $n\geq 6$,
\begin{eqnarray*}
T(n,\F(4),2)=\begin{cases}
D(n,4,2)+1&\text{if $n\equiv 5$, $6$, $7$, $9$, $10$, or $11$ {\rm (mod 12)}, } \\
& \text{$n\not\in\{9,10,11\}$,} \\
D(n,4,2)&\text{if $n\equiv 0$, $1$, $2$, $3$, $4$, or $8$ {\rm (mod 12)}, } \\
& \text{$n\not\in\{8,12,13\}$,}  \\
n-5&\text{if $n\in\{8,9,10,11\}$,} \\
8&\text{if $n=12$,} \\
11&\text{if $n=13$,}
\end{cases}
\end{eqnarray*}
except possibly for
$n\in\{16, 17, 25, 28, 32, 37, 38, 39, 73, 76, 85\}$.
\end{corollary}

\appendix

\section{\boldmath Some maximum $2$-$(n,4,1)$ packings with a leave containing $K_5-e$}

In each case, the edges of the $K_5-e$ in the leave are
${[5]\choose 2}\setminus\{ \{4,5\}\}$.

\subsection{\boldmath The blocks of a maximum $2$-$(10,4,1)$ packing with a leave containing $K_5-e$}
$\{4,5,6,7\}$,
$\{3,7,8,9\}$,
$\{1,6,8,10\}$.

\subsection{\boldmath The blocks of a maximum $2$-$(18,4,1)$ packing with a leave containing $K_5-e$}
\begin{equation*}
\begin{array}{ccccc}
$\{4,8,12,16\}$, &
$\{3,6,7,8\}$, &
$\{3,11,13,16\}$, &
$\{2,9,15,16\}$, &
$\{10,11,12,14\}$, \\
$\{2,7,11,17\}$, &
$\{4,9,13,14\}$, &
$\{1,6,9,17\}$, &
$\{5,13,17,18\}$, &
$\{3,14,15,17\}$, \\
$\{2,8,14,18\}$, &
$\{4,7,10,15\}$, &
$\{2,6,10,13\}$, &
$\{1,8,11,15\}$, &
$\{4,6,11,18\}$, \\
$\{5,8,9,10\}$, &
$\{1,10,16,18\}$, &
$\{5,7,14,16\}$, &
$\{3,9,12,18\}$, &
$\{1,7,12,13\}$, \\
$\{5,6,12,15\}$.
\end{array}
\end{equation*}

\subsection{\boldmath The blocks of a maximum $2$-$(19,4,1)$ packing with a leave containing $K_5-e$}
\begin{equation*}
\begin{array}{ccccc}
$\{8,14,17,18\}$, &
$\{2,9,13,14\}$, &
$\{3,7,12,14\}$, &
$\{1,10,14,19\}$, &
$\{4,5,10,18\}$, \\
$\{4,6,14,16\}$, &
$\{6,11,18,19\}$, &
$\{4,11,13,17\}$, &
$\{3,8,15,19\}$, &
$\{5,12,13,19\}$, \\
$\{1,9,12,18\}$, &
$\{3,13,16,18\}$, &
$\{2,7,15,18\}$, &
$\{3,9,10,17\}$, &
$\{4,7,9,19\}$, \\
$\{2,16,17,19\}$, &
$\{5,6,7,17\}$, &
$\{2,6,8,12\}$, &
$\{10,12,15,16\}$, &
$\{7,8,10,13\}$, \\
$\{5,8,9,16\}$, &
$\{5,11,14,15\}$, &
$\{1,7,11,16\}$, &
$\{1,6,13,15\}$.
\end{array}
\end{equation*}

\subsection{\boldmath The blocks of a maximum $2$-$(20,4,1)$ packing with a leave containing $K_5-e$}
\begin{equation*}
\begin{array}{ccccc}
$\{4,6,16,18\}$, &
$\{3,12,16,20\}$, &
$\{1,10,11,15\}$, &
$\{9,12,14,19\}$, &
$\{2,7,10,12\}$, \\
$\{6,7,15,19\}$, &
$\{9,10,17,18\}$, &
$\{4,9,11,13\}$, &
$\{4,12,15,17\}$, &
$\{4,5,10,19\}$, \\
$\{1,8,12,18\}$, &
$\{3,13,18,19\}$, &
$\{5,8,14,17\}$, &
$\{1,16,17,19\}$, &
$\{1,7,13,14\}$, \\
$\{2,6,13,17\}$, &
$\{11,14,18,20\}$, &
$\{2,8,11,19\}$, &
$\{5,6,11,12\}$, &
$\{5,13,15,20\}$, \\
$\{3,8,9,15\}$, &
$\{8,10,13,16\}$, &
$\{3,7,11,17\}$, &
$\{2,14,15,16\}$, &
$\{3,6,10,14\}$, \\
$\{5,7,9,16\}$, &
$\{4,7,8,20\}$, &
$\{1,6,9,20\}$.
\end{array}
\end{equation*}

\subsection{\boldmath The blocks of a maximum $2$-$(24,4,1)$ packing with a leave containing $K_5-e$}
\begin{equation*}
\begin{array}{ccccc}
$\{12,14,15,18\}$, &
$\{3,6,16,18\}$, &
$\{6,9,10,13\}$, &
$\{5,9,15,22\}$, &
$\{3,9,11,21\}$, \\
$\{4,8,15,19\}$, &
$\{1,18,21,22\}$, &
$\{12,16,17,19\}$, &
$\{11,12,22,23\}$, &
$\{4,9,23,24\}$, \\
$\{4,5,6,12\}$, &
$\{3,10,12,24\}$, &
$\{5,8,21,24\}$, &
$\{6,14,17,21\}$, &
$\{1,8,12,13\}$, \\
$\{6,19,22,24\}$, &
$\{4,16,20,21\}$, &
$\{2,18,19,23\}$, &
$\{1,7,17,23\}$, &
$\{3,17,20,22\}$, \\
$\{1,11,16,24\}$, &
$\{2,13,14,16\}$, &
$\{2,7,10,21\}$, &
$\{5,7,14,20\}$, &
$\{8,10,17,18\}$, \\
$\{13,18,20,24\}$, &
$\{2,9,12,20\}$, &
$\{7,8,16,22\}$, &
$\{3,7,13,19\}$, &
$\{2,15,17,24\}$, \\
$\{5,11,13,17\}$, &
$\{13,15,21,23\}$, &
$\{10,11,19,20\}$, &
$\{1,9,14,19\}$, &
$\{4,7,11,18\}$, \\
$\{1,6,15,20\}$, &
$\{3,8,14,23\}$, &
$\{2,6,8,11\}$, &
$\{5,10,16,23\}$, &
$\{4,10,14,22\}$.
\end{array}
\end{equation*}

\subsection{\boldmath The blocks of a maximum $2$-$(26,4,1)$ packing with a leave containing $K_5-e$}
\begin{equation*}
\begin{array}{ccccc}
$\{4,17,22,24\}$, &
$\{3,11,17,20\}$, &
$\{5,7,18,22\}$, &
$\{4,16,18,23\}$, &
$\{1,7,19,25\}$, \\
$\{14,21,22,23\}$, &
$\{1,10,18,26\}$, &
$\{2,11,21,26\}$, &
$\{3,6,7,23\}$, &
$\{11,14,16,19\}$, \\
$\{12,20,24,26\}$, &
$\{4,7,14,26\}$, &
$\{3,9,16,22\}$, &
$\{6,10,15,16\}$, &
$\{3,10,12,19\}$, \\
$\{7,8,15,17\}$, &
$\{4,9,13,19\}$, &
$\{5,12,13,21\}$, &
$\{15,19,22,26\}$, &
$\{5,19,23,24\}$, \\
$\{4,12,15,25\}$, &
$\{3,15,18,21\}$, &
$\{8,9,21,25\}$, &
$\{6,12,17,18\}$, &
$\{5,8,16,26\}$, \\
$\{2,7,9,12\}$, &
$\{9,17,23,26\}$, &
$\{1,8,20,22\}$, &
$\{5,9,11,15\}$, &
$\{7,10,21,24\}$, \\
$\{1,13,14,15\}$, &
$\{6,19,20,21\}$, &
$\{7,13,16,20\}$, &
$\{10,11,22,25\}$, &
$\{2,6,13,22\}$, \\
$\{2,16,24,25\}$, &
$\{9,14,18,20\}$, &
$\{2,8,18,19\}$, &
$\{1,6,9,24\}$, &
$\{4,6,8,11\}$, \\
$\{5,6,14,25\}$, &
$\{8,10,13,23\}$, &
$\{11,13,18,24\}$, &
$\{2,10,14,17\}$, &
$\{3,13,25,26\}$, \\
$\{3,8,14,24\}$, &
$\{2,15,20,23\}$, &
$\{1,11,12,23\}$, &
$\{4,5,10,20\}$, &
$\{1,16,17,21\}$.
\end{array}
\end{equation*}

\subsection{\boldmath The blocks of a maximum $2$-$(27,4,1)$ packing with a leave containing $K_5-e$}
\begin{equation*}
\begin{array}{ccccc}
$\{2,7,16,21\}$, &
$\{7,20,26,27\}$, &
$\{5,17,25,27\}$, &
$\{5,15,21,23\}$, &
$\{5,6,11,22\}$, \\
$\{13,21,22,27\}$, & 
$\{3,8,11,26\}$, &
$\{6,15,17,24\}$, &
$\{4,5,7,19\}$, &
$\{1,6,18,27\}$, \\
$\{3,18,21,24\}$, &
$\{2,11,12,13\}$, &
$\{9,13,16,23\}$, &
$\{10,11,14,15\}$, &
$\{3,14,23,27\}$, \\
$\{4,8,15,18\}$, &
$\{14,19,22,24\}$, &
$\{1,10,19,23\}$, &
$\{3,12,16,20\}$, &
$\{2,8,23,24\}$, \\
$\{5,8,9,20\}$, &
$\{4,12,14,21\}$, &
$\{4,9,11,27\}$, &
$\{3,6,10,25\}$, &
$\{8,14,16,17\}$, \\
$\{2,15,19,27\}$, &
$\{9,12,15,22\}$, &
$\{3,7,13,15\}$, &
$\{1,8,12,25\}$, &
$\{3,9,17,19\}$, \\
$\{19,20,21,25\}$, & 
$\{2,6,14,20\}$, &
$\{6,8,13,19\}$, &
$\{7,11,24,25\}$, &
$\{1,11,17,21\}$, \\
$\{4,10,17,20\}$, &
$\{9,10,21,26\}$, &
$\{10,16,24,27\}$, &
$\{4,16,22,25\}$, &
$\{7,12,17,18\}$, \\
$\{1,7,9,14\}$, &
$\{2,17,22,26\}$, &
$\{11,16,18,19\}$, &
$\{5,12,24,26\}$, &
$\{1,15,16,26\}$, \\
$\{5,10,13,18\}$, &
$\{1,13,20,24\}$, &
$\{18,20,22,23\}$, &
$\{2,9,18,25\}$, &
$\{4,6,23,26\}$, \\
$\{13,14,25,26\}$, & 
$\{7,8,10,22\}$.
\end{array}
\end{equation*}

\subsection{\boldmath The blocks of a maximum $2$-$(36,4,1)$ packing with a leave containing $K_5-e$}
\begin{equation*}
\begin{array}{ccccc}
$\{7,10,17,35\}$, &
$\{11,15,26,36\}$, &
$\{6,16,25,29\}$, &
$\{1,12,24,28\}$, &
$\{3,13,34,35\}$, \\
$\{30,31,35,36\}$, &
$\{21,23,28,34\}$, &
$\{1,14,19,35\}$, &
$\{8,9,28,32\}$, &
$\{15,18,21,25\}$, \\
$\{3,18,26,27\}$, &
$\{1,8,25,30\}$, &
$\{3,10,23,29\}$, &
$\{6,28,30,33\}$, &
$\{15,19,23,24\}$, \\
$\{4,14,17,34\}$, &
$\{7,13,26,28\}$, &
$\{10,19,22,36\}$, &
$\{6,11,12,34\}$, &
$\{1,7,11,29\}$, \\
$\{5,13,17,25\}$, &
$\{14,24,26,31\}$, &
$\{13,19,27,29\}$, &
$\{1,20,23,26\}$, &
$\{2,22,31,34\}$, \\
$\{14,23,25,36\}$, &
$\{5,16,24,33\}$, &
$\{4,18,29,33\}$, &
$\{4,21,26,32\}$, &
$\{8,22,26,29\}$, \\
$\{9,11,22,25\}$, &
$\{12,18,20,32\}$, &
$\{2,11,20,21\}$, &
$\{11,13,31,32\}$, &
$\{10,14,30,32\}$, \\
$\{3,9,33,36\}$, &
$\{3,11,24,30\}$, &
$\{24,29,32,36\}$, &
$\{7,18,24,34\}$, &
$\{7,19,21,31\}$, \\
$\{3,7,12,25\}$, &
$\{2,8,13,24\}$, &
$\{2,7,14,16\}$, &
$\{5,7,8,20\}$, &
$\{10,11,16,28\}$, \\
$\{5,6,18,31\}$, &
$\{8,11,14,18\}$, &
$\{3,17,19,32\}$, &
$\{10,20,25,31\}$, &
$\{4,5,11,19\}$, \\
$\{16,18,19,30\}$, &
$\{16,20,34,36\}$, &
$\{3,6,15,20\}$, &
$\{4,8,10,12\}$, &
$\{6,9,13,14\}$, \\
$\{9,17,20,24\}$, &
$\{13,20,22,33\}$, &
$\{4,6,7,36\}$, &
$\{1,13,18,36\}$, &
$\{5,26,30,34\}$, \\
$\{1,6,22,32\}$, &
$\{16,21,27,35\}$, &
$\{12,13,21,30\}$, &
$\{2,9,18,35\}$, &
$\{12,17,29,31\}$, \\
$\{8,17,21,36\}$, &
$\{7,9,23,30\}$, &
$\{20,28,29,35\}$, &
$\{2,15,29,30\}$, &
$\{4,20,27,30\}$, \\
$\{1,15,16,17\}$, &
$\{1,9,27,31\}$, &
$\{4,15,28,31\}$, &
$\{12,14,15,33\}$, &
$\{9,12,19,26\}$, \\
$\{25,26,33,35\}$, &
$\{6,10,24,27\}$, &
$\{3,14,21,22\}$, &
$\{2,23,32,33\}$, &
$\{5,14,27,28\}$, \\
$\{5,12,22,23\}$, &
$\{25,27,32,34\}$, &
$\{3,8,16,31\}$, &
$\{4,13,16,23\}$, &
$\{8,19,33,34\}$, \\
$\{2,6,17,26\}$, &
$\{5,15,32,35\}$, &
$\{2,19,25,28\}$, &
$\{9,10,15,34\}$, &
$\{6,8,23,35\}$, \\
$\{1,10,21,33\}$, &
$\{7,15,22,27\}$, &
$\{4,22,24,35\}$, &
$\{11,17,27,33\}$, &
$\{2,12,27,36\}$, \\
$\{5,9,21,29\}$, &
$\{17,18,22,28\}$.
\end{array}
\end{equation*}

\section{\boldmath Some maximum $2$-$(n,4,1)$ packings with a leave containing $2\circ K_4$}

\subsection{\boldmath The blocks of a maximum $2$-$(18,4,1)$ packing with a leave containing $2\circ K_4$}
\begin{equation*}
\begin{array}{ccccc}
$\{3,9,10,12\}$, &
$\{1,7,9,16\}$, &
$\{4,7,17,18\}$, &
$\{1,5,13,17\}$, &
$\{5,9,11,14\}$, \\
$\{1,6,14,18\}$, &
$\{4,6,8,9\}$, &
$\{2,7,10,14\}$, &
$\{3,14,16,17\}$, &
$\{5,8,10,18\}$, \\
$\{1,8,12,15\}$, &
$\{6,10,15,17\}$, &
$\{3,7,8,13\}$, &
$\{3,11,15,18\}$, &
$\{6,7,11,12\}$, \\
$\{2,12,16,18\}$, &
$\{10,11,13,16\}$, &
$\{2,8,11,17\}$, &
$\{2,9,13,15\}$, &
$\{4,5,15,16\}$, \\
$\{4,12,13,14\}$.
\end{array}
\end{equation*}

\subsection{\boldmath The blocks of a maximum $2$-$(19,4,1)$ packing with a leave containing $2\circ K_4$}
\begin{equation*}
\begin{array}{ccccc}
$\{3,8,9,16\}$, &
$\{5,12,16,18\}$, & 
$\{11,13,15,16\}$, &
$\{1,12,13,19\}$, &
$\{6,8,10,11\}$, \\
$\{1,10,14,16\}$, &
$\{6,12,15,17\}$, &
$\{6,7,9,13\}$, &
$\{4,13,14,18\}$, &
$\{1,9,15,18\}$, \\
$\{5,9,14,19\}$, &
$\{4,6,16,19\}$, &
$\{4,7,8,12\}$, &
$\{5,8,13,17\}$, &
$\{2,8,14,15\}$, \\
$\{3,10,17,18\}$, &
$\{4,5,10,15\}$, &
$\{2,9,10,12\}$, &
$\{1,5,7,11\}$, &
$\{2,7,16,17\}$, \\
$\{4,9,11,17\}$, &
$\{3,7,15,19\}$, &
$\{3,11,12,14\}$, &
$\{2,11,18,19\}$.
\end{array}
\end{equation*}

\subsection{\boldmath The blocks of a maximum $2$-$(24,4,1)$ packing with a leave containing $2\circ K_4$}
\begin{equation*}
\begin{array}{ccccc}
$\{4,13,14,21\}$, &
$\{3,12,16,20\}$, &
$\{3,15,21,22\}$, &
$\{3,17,18,23\}$, &
$\{7,13,15,23\}$, \\
$\{4,12,19,22\}$, &
$\{1,9,15,19\}$, &
$\{4,7,8,18\}$, &
$\{6,10,15,18\}$, &
$\{9,11,17,21\}$, \\
$\{3,10,11,13\}$, &
$\{5,9,14,22\}$, &
$\{1,11,14,18\}$, &
$\{1,6,8,21\}$, &
$\{6,7,14,20\}$, \\
$\{2,13,19,24\}$, &
$\{10,19,20,21\}$, &
$\{5,11,12,15\}$, &
$\{8,11,16,24\}$, &
$\{2,8,10,17\}$, \\
$\{2,16,21,23\}$, &
$\{5,8,13,20\}$, &
$\{4,6,9,16\}$, &
$\{1,7,10,16\}$, &
$\{2,7,11,22\}$, \\
$\{2,9,18,20\}$, &
$\{14,15,16,17\}$, &
$\{8,9,12,23\}$, &
$\{10,12,14,24\}$, &
$\{3,7,9,24\}$, \\
$\{13,16,18,22\}$, &
$\{6,17,22,24\}$, &
$\{1,12,13,17\}$, &
$\{5,7,17,19\}$, &
$\{3,8,14,19\}$, \\
$\{4,15,20,24\}$, &
$\{1,20,22,23\}$, &
$\{4,5,10,23\}$, &
$\{6,11,19,23\}$, &
$\{5,18,21,24\}$.
\end{array}
\end{equation*}

\subsection{\boldmath The blocks of a maximum $2$-$(27,4,1)$ packing with a leave containing $2\circ K_4$}
\begin{equation*}
\begin{array}{ccccc}
$\{6,12,17,21\}$, &
$\{1,5,10,19\}$, &
$\{4,8,10,27\}$, &
$\{4,14,21,22\}$, &
$\{19,21,24,25\}$, \\
$\{2,11,23,26\}$, &
$\{3,10,20,24\}$, &
$\{3,9,15,21\}$, &
$\{12,20,26,27\}$, &
$\{8,11,12,25\}$, \\
$\{10,15,18,26\}$, &
$\{3,8,19,26\}$, &
$\{4,11,13,20\}$, &
$\{9,22,24,26\}$, &
$\{3,7,22,27\}$, \\
$\{1,15,22,23\}$, &
$\{5,14,24,27\}$, &
$\{3,12,14,23\}$, &
$\{9,12,13,19\}$, &
$\{2,9,17,27\}$, \\
$\{3,13,18,25\}$, &
$\{4,7,12,15\}$, &
$\{6,14,16,20\}$, &
$\{5,7,9,11\}$, &
$\{6,15,25,27\}$, \\
$\{10,17,22,25\}$, &
$\{5,20,23,25\}$, &
$\{2,10,12,16\}$, &
$\{4,6,18,19\}$, &
$\{5,12,18,22\}$, \\
$\{3,11,16,17\}$, &
$\{6,8,13,22\}$, &
$\{1,13,16,27\}$, &
$\{2,13,15,24\}$, &
$\{5,8,15,17\}$, \\
$\{5,16,21,26\}$, &
$\{13,14,17,26\}$, &
$\{7,10,13,21\}$, &
$\{2,19,20,22\}$, &
$\{1,6,11,24\}$, \\
$\{7,17,18,20\}$, &
$\{1,8,20,21\}$, &
$\{1,7,25,26\}$, &
$\{11,14,15,19\}$, &
$\{4,9,16,25\}$, \\
$\{7,16,19,23\}$, &
$\{8,16,18,24\}$, &
$\{11,18,21,27\}$, &
$\{4,17,23,24\}$, &
$\{1,9,14,18\}$, \\
$\{2,7,8,14\}$, &
$\{6,9,10,23\}$.
\end{array}
\end{equation*}\pagebreak


\end{document}